\documentclass[a4paper,11pt]{article}
\usepackage[english]{babel}
\usepackage{cite} 
\usepackage[dvips]{graphicx}

\usepackage[usenames, dvipsnames]{xcolor}
\usepackage[utf8x]{inputenc}
\usepackage[T1]{fontenc}
\usepackage{tikz}
\usepackage{pgfplots}

\usepackage[backref=page]{hyperref}

\usepackage{amsfonts,amsmath,amsxtra,amsthm,amssymb,latexsym}
\DeclareMathOperator{\sech}{sech}

\newtheorem{theorem}{Theorem}[section]
\newtheorem{corollary}[theorem]{Corollary}
\newtheorem{lemma}[theorem]{Lemma}

\theoremstyle{definition}
\newtheorem{definition}[theorem]{Definition}
\theoremstyle{remark}
\newtheorem{remark}[theorem]{Remark}

\numberwithin{equation}{section}

\newcommand{\R}{\mathbb R}
\begin{author}
	{Jaime Angulo Pava $^1$ and M\'arcio Cavalcante $^2$}
\end{author}

\begin{title}
	{Dynamics of the Korteweg-de Vries equation on a balanced metric graph}
\end{title}
\date{}
\begin{document}
	\maketitle
	
	\centerline{$^1$ Department of Mathematics,
		IME-USP}
	\centerline{Rua do Mat\~ao 1010, Cidade Universit\'aria, CEP 05508-090,
		S\~ao Paulo, SP, Brazil.}
	\centerline{\it angulo@ime.usp.
		br}
	
	\centerline{ $^2$ Institute of Mathematics,  Universidade Federal de Alagoas,
		Macei\'o, Brazil.}
	\centerline{\it  marcio.melo@im.ufal.br }

	\begin{abstract}
		In this work, we establish local well-posedness for the Korteweg-de Vries model on a balanced star graph with a structure represented by semi-infinite edges, by considering a boundary condition of $\delta$-type at the   {unique} graph-vertex. Also, we extend the linear instability result in 
 \cite{AC} to  one of nonlinear instability. For the proof of local well posedness theory the principal new ingredient is the utilization of the special solutions by Faminskii in the context of half-lines. As far as we are aware, this approach is being used for the first time in the context of star graphs and can potentially be applied to other boundary classes. In the case of the nonlinear instability result, the principal ingredients  are the linearized instability known result and the fact that data-to-solution map determined by the local theory is at least of class $C^2$.

	\end{abstract}

	\let\thefootnote\relax\footnotetext{
		\text{Key  words}. Korteweg-de Vries model, star graph, bumps, $\delta$-type, nonlinear  instability.}
	
	\section{Introduction}

The focus of our study to follow will be  the well-known Korteweg-de Vries equation  (KdV henceforth)
\begin{equation}\label{kdv}
	\partial_t u +\partial_x^3u+ \partial_x u+ 2u\partial_xu = 0 
	\end{equation}
	on a metric graph of balanced type (see Figure 2 below). We recall that  an evolution model on a metric graph  is equivalent to a system of PDEs defined on  the edges (intervals) with a  coupling  given exclusively through the boundary conditions at the vertices (known as the ``topology of the graph'') and  which will determine the dynamic on the network or graph. Moreover, the freedom in setting the topology in a graph  allows us to create different dynamics much closer to the real world applications.
\\	
We recall that the KdV was first derived by Korteweg and de-Vries \cite{KdV} in 1895 as a model for long waves propagating on a shallow water surface as well as provided an explanation of the existence of {\it the Great Wave of Translation} a phenomena first discovered by Scott Russell in 1834 \cite{Russell}. Actually, this  type  of solutions are known as solitary wave solutions  or soliton profiles (Boussinesq \cite{bous}, Kruskal and Zabusky \cite{Zabusky}). The KdV equation also arise naturally in the modeling of various types of wave phenomena in other physical context, such as, the nonlinear mass-spring system (FPU recurrence phenomenon), ion-acoustic waves in a collisionless plasma and magnetosonic waves in a magnetized plasma (see \cite{AbloCla} and the references therein).  In particular,  in 1987, Sigeo \cite{Sigeo} took the lead in combining nonlinear science with hemodynamics and derived the KdV equation for the velocity of blood flow. Also,  in  \cite{Crepeau} the KdV equation has been used  as a model to study blood pressure waves in large arteries.
	
Because of the range of its potential application, dynamics of the KdV equation from a mathematical context has been well studied in the last decades and a  marvelous quantity of manuscripts have been published. These studies have generally concentrated on aspects of the pure initial-value problem, namely, $u=u(x,t)$, $x\in \mathbb R$ and $t\geqq 0$, satisfying \eqref{kdv} with the initial condition $u(x,0)=f(x)$. Thus, the Cauchy problem for the KdV posed on the real axis, on the torus, on the half-lines and on a finite interval has been studied comprehensively (see \cite{b1,CK,Faminskii,Holmer,Jia,KPV2,Kishimoto,Killip} and the references therein). Also, closely related to these studies is the large-time asymptotic behavior of solutions of the KdV equation close to localized coherent structures such as the solitary waves solutions, namely, solutions of the KdV equation with the traveling-wave profile $u_s(x,t)=\phi_c(x-ct)$, $c>1$, where the profile $\phi_c$ is given by 
\begin{equation}\label{solitary}
\phi_c(\xi)= \frac{3}{2}(c-1)\sech^2\Big(\frac{\sqrt{c-1}}{2} \xi \Big),\quad \xi\in \mathbb R.
\end{equation}
The principal study associated to the wave profiles $\phi_c$ is the so-called {\it stability in shape or orbital stability},  namely, a slight perturbation of the profile $\phi_c$ will continue to resemble a solitary wave all of the time, rather than evolving into some other wave form (see Boussinesq \cite{bous}, Benjamin \cite{benjamin}, Bona,  P.  Souganidis and  Strauss \cite{bss}, Pego and  Weinstein \cite{Pego}. See also \cite{Angbook} for a comprehensive description of these results).

Studies for the KdV equation on networks or branched  structures has drawn attention in recent years. In Ammari and Crepeau \cite {Am} (see also \cite{Cerpa}) was established a control theory for the KdV  equation on a finite star-shaped network  which is in connection with the mathematical modeling of the human cardiovascular system, namely,  they considered a system formed by $N$-KdV equations in  variable $u_j$  in \eqref{kdv} posed on  bounded intervals $(0, L_j)$, $j=1,..., N$, and with specific boundary conditions at the vertex-graph $\nu=0$ and on the external nodes $L_j$. Moreover, recently in Cavalcante \cite{Cav1} was studied the  local well-posedness problem for the KdV  equation in Sobolev spaces $H^s(\mathcal Y)$ with low regularity on a $\mathcal Y$-junction graph with three semi-infinite edges given by one negative half-line and two positive half-lines attached to a common vertex $\nu=0$.

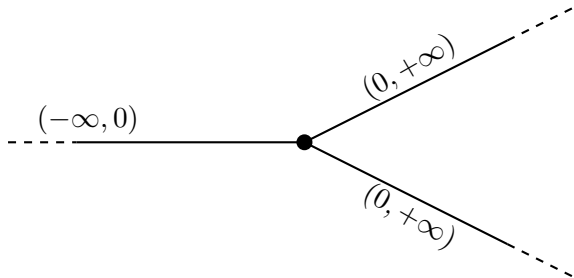
\begin{figure}[htp]\label{figure2}
		\centering 
		\begin{tikzpicture}[scale=3]
			\draw [thick, dashed] (-1.3,0)--(-1,0);
			\draw[thick](-1,0)--(0,0);
			\node at (-0.95,0.1){$(-\infty,0)$};
			
			
			

			\draw[thick](0,0)--(0.89,0.45);
			\node at (0.46,0.33)[rotate=3
			0]{$(0,+\infty)$};
			\draw [thick, dashed] (0.89,0.45)--(1.19,0.6);
			\draw[thick](0,0)--(0.89,-0.45);
			\node at (0.46,-0.33)[rotate=-30]{$(0,+\infty)$};
			\draw [thick, dashed] (0.89,-0.45)--(1.19,-0.6);
			\fill (0,0)  circle[radius=1pt];
		\end{tikzpicture}
		\caption{A $\mathcal Y$-junction graph}
	\end{figure}

Very recently, Mugnolo, Noja and Seifert \cite{MNS} obtained a characterization of all boundary conditions under which the
	Airy-type evolution equation 
	\begin{equation}\label{kdv0}
		\partial_{t}u_\bold e(x,t)=\alpha_ \bold e \partial_x ^3u_\bold e(x,t) + \beta_ \bold e \partial_x u_\bold e(x,t), \; \; t\in \mathbb{R},\; x=x_\bold e \in \bold e, \; \bold e\in \bold E,
	\end{equation}
	generates either a semigroup or a unitary group on a metric star graph $\mathcal G$. Here $\mathcal G$ is a structure represented by the set $
	\bold E\equiv \bold E_{-}\cup \bold E_{+}$
	where $\bold E_{+}$ and $\bold E_{-}$ are finite or countable collections of semi-infinite edges $\bold e$ parametrized by $(-\infty, 0)$ or $(0, +\infty)$, respectively. The half-lines are connected at a unique vertex $\nu=0$. 
In \eqref{kdv0} we are using  $u_\bold e (x,t)$ in the sense that $u_\bold e: I\times [0,T]\to \mathbb R$  for $x=x_{\bold e} \in I$ and $I$ is the half-line  determined by $\bold e$, and by abusing notation, we are using $x_\bold e \in \bold e$. Here  $(\alpha_ \bold e)_{\bold e\in \bold E}$ and  $(\beta_ \bold e)_{\bold e\in \bold E}$ are two sequences of real numbers.

Thus, one of the objectives of this work is to shed light on  the local well-posedness problem for the KdV  equation 
on a metric graph of balanced type, in other words,  we are interested in the case of the Airy operator
	\begin{equation}
		A_0: (u_{\bold e})_{\bold e \in \bold E}	\mapsto \left(\alpha_{\bold e}\frac{d}{dx^3}u_{\bold e}+\beta_{\bold e}\frac{d}{dx}u_{\bold e}\right)_{\bold e \in \bold E}
	\end{equation}
	defined on a metric graph $\mathcal G$ when the ingoing half-lines equal to outgoing half-lines ( $|\bold E_{+}|=|\bold E_{-}|= n$, see Figure 2). Thus, from  \cite{MNS} there many  possible skew-symmetric extensions $A_{ext}$ of $A_0$ and so  by Stone's theorem the solution of the following linear equation
	\begin{equation}\label{Alinear}
	\left\{ \begin{array}{ll}
			&z_t=A_{ext}z, \qquad t\in \mathbb R\\
			&z(0)=z_0\in D(A_{ext})
		\end{array}  \right.
	\end{equation}
	will be  given by a $C_0$-unitary group $z(t)=e^{tA_{ext}}z_0$.
	
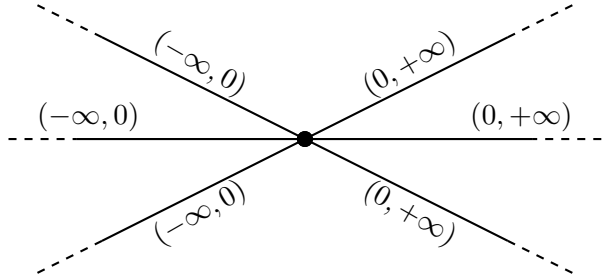
\begin{figure}[htp]\label{figure2}
		\centering 
		\begin{tikzpicture}[scale=3]
			\draw [thick, dashed] (-1.3,0)--(-1,0);
			\draw[thick](-1,0)--(0,0);
			\node at (-0.95,0.1){$(-\infty,0)$};
			
			\draw [thick, dashed] (1.3,0)--(1,0);
			\draw[thick](1,0)--(0,0);
			\node at (0.95,0.1){$(0,+\infty)$};
			
			\draw[thick](0,0)--(-0.89,0.45);
			\node at (-0.46,0.33)[rotate=-3
			0]{$(-\infty,0)$};
			\draw [thick, dashed] (-0.89,0.45)--(-1.19,0.6);
			
			\draw[thick](0,0)--(-0.89,-0.45);
			\node at (-0.46,-0.33)[rotate=30]{$(-\infty,0)$};
			\draw [thick, dashed] (-0.89,-0.45)--(-1.19,-0.6);
			\fill (0,0)  circle[radius=1pt];

			\draw[thick](0,0)--(0.89,0.45);
			\node at (0.46,0.33)[rotate=3
			0]{$(0,+\infty)$};
			\draw [thick, dashed] (0.89,0.45)--(1.19,0.6);
			\draw[thick](0,0)--(0.89,-0.45);
			\node at (0.46,-0.33)[rotate=-30]{$(0,+\infty)$};
			\draw [thick, dashed] (0.89,-0.45)--(1.19,-0.6);
			\fill (0,0)  circle[radius=1pt];
		\end{tikzpicture}
		\caption{A balanced star graph with $6$ edges}
	\end{figure}

	In the following we defined the pair $(A_{ext}, D(A_{ext}))$ which will  be of our interest  here.  By convenience of the reader, we start with some basic notation. For $\bold{u}=(u_{\bold e})_{\bold e \in \bold E}$, we denote $\bold{u}(0-)$ and $\bold{u}(0+)$ as
	$$
	\bold{u}(0-)=(u_{\bold e}(0-))_{\bold e \in \bold E_{-}} \ \text{and}\ \bold{u}(0+)\equiv(u_{\bold e}(0+))_{\bold e\in \bold{E}_+}.
	$$
	Similarly, we define $\bold{u}'(0-)$, $\bold{u}'(0+)$, $\bold{u}''(0-)$, and $\bold{u}''(0+)$. Therefore,  for $(A_0, D(A_0))$ with
	$$
	D(A_0)\equiv \bigoplus\limits_{\bold e\in  \bold E_{-} }C_c ^{\infty}(-\infty, 0)  \oplus \bigoplus\limits_{\bold e\in  \bold E_{+} }C_c ^{\infty}(0, +\infty),
	$$
	$iA_0$ is a densely defined  symmetric operator on  the Hilbert space $L^2(\mathcal{G})$, and for $Z\in \mathbb R-\{0\}$  we obtain a family $(A_Z, D(A_Z))$ of skew-self-adjoint extension of $(A_0, D(A_0))$ parametrized by $Z$, where
	\begin{equation}\label{domain8}
		\left\{ \begin{array}{ll}
			A_Z\bold{u}&= \Big (\alpha_ \bold e \frac{d^3}{dx^3}u_\bold e + \beta_ \bold e \frac{d}{dx} u_\bold e\Big)_{\bold e\in \bold E}, \qquad \bold{u}=(u_\bold e)_{\bold e\in \bold E}\\
			\\
			D(A_Z) &= \{\bold{u}\in H^3(\mathcal G): \bold{u}{}(0-)=\bold{u}(0+),\; \;\;\bold{u}'(0+)- \bold{u}'(0-)=Z\bold{u}(0-),\\
			&\hskip0.7in \frac{Z^2}{2}\bold{u}(0-)+Z\bold{u}'(0-)=\bold{u}''(0+)-\bold{u}''(0-)\},
		\end{array}  \right.
	\end{equation}
	with 
	$$
	H^3(\mathcal{G})=\bigoplus\limits_{\bold e\in  \bold E_{-} }H^3(-\infty, 0) \oplus \bigoplus\limits_{\bold e\in  \bold E_{+} } H^3(0, +\infty).
	$$
	Moreover,  for $(\alpha_ \bold e)_{\bold e\in \bold E}=(\alpha_{-}, \alpha_{+})$ ( $\alpha_{-}=(\alpha_ \bold e)_{\bold e\in \bold E_{-}}$, $\alpha_{+}=(\alpha_ \bold e)_{\bold e\in \bold E_{+}}$) and   $(\beta_ \bold e)_{\bold e\in \bold E}=(\beta_{-}, \beta_{+})$ we need to have $\alpha_{-}= \alpha_{+}$ and $\beta_{-}=\beta_{+}$. In particular, the system of boundary condition at the vertex $\nu=0$ of the graph $\mathcal G$,
	$$
	\bold{u}{}(0-)=\bold{u}(0+)\ \;\; \text{and}\ \;\; \bold{u}'(0+)- \bold{u}'(0-)=Z\bold{u}(0-)
	$$
	are called {\it $\delta$-interaction type on each two oriented half-lines}. A formula for the unitary-group generated by $(A_Z, D(A_Z))$ was obtained in \cite{AC}. We recall that as $iA_0$ is a symmetric operator with deficiency indices $(3n,3n)$, it follows from the classical von Neumann-Krein extension theory that the operator $(A_0, D(A_0))$ will admit a $9n^2$-parameter family of skew-self-adjoint extensions generating each one a unitary dynamics on $L^2(\mathcal G)$ associated to the linear problem \eqref{Alinear}. Our interest in the operator $A_Z$ defined above is due to the study of the non-linear instability of specific stationary solutions for the KdV with a profile in the domain $D(A_Z)$ (see Angulo and Cavalcante \cite{AC}).
	
The discussion is now turned more directly to the contributions in the present manuscript. We commence with the Cauchy problem. We note first that a space-time function $\bold{u}(x,t)$ on $\mathcal G \times [0,T]$ will be denoted as $\bold{u}(x,t)=(u_\bold e (x,t))_{\bold e\in \bold E}$. Thus, we are interested  in the local well-posedness theory for the following Cauchy problem for the KdV model on $\mathcal G$,
	\begin{equation}\label{kdv3}
		\begin{cases}
			\partial_{t}u_\bold e(x,t)=\alpha_ \bold e \partial_x ^3u_\bold e(x,t) + \beta_ \bold e \partial_xu_\bold e(x,t)+2 u_\bold e (x,t) \partial_x u_\bold e (x,t),\;\; \bold e\in \bold E,\ t\in (0,T) \\
			\bold{u} (x,0) = \bold{u}_0 (x) \in H^s(\mathcal{G})\cap \mathcal N_{0, Z}(s),
		\end{cases}
	\end{equation}
	with $s\geqq 1$, and $ \bold{u}$ satisfying,
		\begin{equation}\label{Zcondition}
			\begin{cases}
				\bold{u}(0-, t)=\bold{u}(0+,t), \;\;\text{in\ the\ sense\ of } H^{\frac{s+1}{3}}(0,T);\\ 
				\bold{u}'(0+,t)-\bold{u}'(0-,t)=Z \bold{u}(0-, t),\;\; \text{in\ the\ sense\ of } H^{\frac{s}{3}}(0,T);\\  
				\frac{Z^2}{2}\bold{u}(0-,t)+Z\bold{u}'(0-,t)=\bold{u}''(0+,t)-\bold{u}''(0-,t),\;\;\text{in\ the\ sense\ of } H^{\frac{s-1}{3}}(0,T),
			\end{cases}
		\end{equation}
{moreover}  the set $\mathcal N_{0, Z}(s)$ determines the following compatibility conditions on the vertex $\nu=0$,
	\begin{equation}\label{trace0}
	\mathcal N_{0, Z}(s)=	\begin{cases}
			\{ \bold{v}: \mathcal{G}\to \mathbb R:  \bold{v}(0-)=\bold{v}(0+)\}, & \text{if} \;\; s\in [1, \frac32]\\
			\{ \bold{v}: \mathcal{G}\to \mathbb R:   \bold{v}(0-)=\bold{v}(0+), \;\bold{v}'(0+)-\bold{v}'(0-)=Z \bold{v}(0-)\}, & \text{if} \;\;  s\in (\frac32, \frac52]\\
			\{ \bold{v}: \mathcal{G}\to \mathbb R:   \bold{v}(0-)=\bold{v}(0+), \;\bold{v}'(0+)-\bold{v}'(0-)=Z \bold{v}(0-),\\
			\hskip1.3in \frac{Z^2}{2}\bold{v}(0-)+Z\bold{v}'(0-)=\bold{v}''(0+)-\bold{v}''(0-)\}, &\text{if} \;\;  s>\frac52\\
		\end{cases}
	\end{equation}
	
We note by \eqref{domain8}, that $D(A_Z)\subset  \mathcal N_{0, Z}(s)$, for all $s\geqq 1$.
	
	For solving the Cauchy problem \eqref{kdv3}-\eqref{Zcondition} we will consider the following functional space $X_{s,b,\beta,\sigma}(T)$, $T>0$, $s\geqq 1$,
	\begin{equation}\label{X}
	\begin{aligned}
		X_{s,b,\beta,\sigma}(T)=&\left\{\bold{w}: \mathcal{G}\times [0,T]\to \mathbb R: \bold w \in C\left([0,T] ; H^s\left(\mathcal G\right)\right) \cap C\left(\mathcal G ; H^{\frac{s+1}{3}}\left([0,T]\right)\right) \cap \bold {\mathcal X}^{s, b, \beta, \sigma},\right. \\
		&\quad \quad \quad \left. \partial_x \bold{w} \in C\left(\mathcal G ; H^{\frac{s}{3}}\left([0,T]\right)\right), \partial_x^2\bold{w} \in C\left(\mathcal G ; H^{\frac{s-1}{3}}\left([0,T]\right)\right)\right\},
	\end{aligned}
	\end{equation}
	with $\bold {\mathcal X}^{s, b, \beta, \sigma}$ being the local Bourgain's space in \eqref{Bour}, and for $r\geqq 0$ we define $C\left(\mathcal G ; H^r\left([0,T]\right)\right) $ as
	\begin{equation}\label{trace2}
		\begin{aligned}
			C\left(\mathcal G ; H^r\left([0,T]\right)\right) =\{\bold{u}: \mathcal G \times [0,T]&\to \mathbb R: \; \text{for}\; \bold{u}=(u_\bold e (x,t))_{\bold e\in \bold E}\;\text{we have for each fixed}\; x\\
			&\text{and}\; \bold e\in \bold E, \; u_\bold e (x,\cdot)\in  H^r\left([0,T])\right\},
		\end{aligned}
	\end{equation}
	with the norm given by
$$
{\| \bold u\|_{C\left(\mathcal G ; H^{r}\left([0,T]\right)\right) }=\sum_{\bold e \in \bold E^+}}\|u_{\bold e}\|_{C(\R^+;H^s(0,T))}+\sum_{\bold e \in \bold E^-}\|u_{\bold e}\|_{C(\R^-;H^s(0,T))}.
$$	

\begin{remark}
In the spirit of Faminskii works \cite{Faminskii} and \cite{Faminskii2} on the context of half-lines, we get the required estimates on the space $X_{s,b,\beta,\sigma}\subset C([0,T];H^s)$ for some $(s,b,\beta,\sigma)$ satisfying  $s>0$, $b<\frac12$, $\sigma^\frac12$. In order to shorten the notation, in this paper we sometimes denote the space $X_{s,b,\beta,\sigma}$ only by $X_{s}$.
\end{remark}
We note that $X_s(T)$ becomes a Banach space with the norm
	\begin{equation*}
	\begin{split}
	\|\bold w\|_{X_s(T)}:=\|\bold w&\|_{C\left([0,T] ; H^s\left(\mathcal G\right)\right)}+\| \bold w\|_{C\left(\mathcal G ; H^{\frac{s+1}{3}}\left([0,T]\right)\right) }+\|\partial_x \bold w\|_{C\left(\mathcal G ; H^{\frac{s}{3}}\left([0,T]\right)\right) }\\
	&+\|\partial_x^2 \bold w\|_{C\left(\mathcal G ; H^{\frac{s-1}{3}}\left([0,T]\right)\right) }+\|\bold w\|_{\bold {\mathcal X}^{s, b, \beta, \sigma}}.
	\end{split}
	\end{equation*}

	Our   local well-posedness theory for the  Cauchy problem in \eqref{kdv3}-\eqref{Zcondition}  is the following,
	
	\begin{theorem}\label{theorem1}
		Let $s\geqq 1$, $Z\neq 0$, and $\bold u_0 \in H^s(\mathcal{G})\cap  \mathcal N_{0, Z}(s)$. Then there exists $T=T(\|\bold u_0\|_s)$ and a unique solution $\bold u$ of the IVP \eqref{kdv3}-\eqref{Zcondition}  in the class $X_s(T)$ such that $\bold u(0)=\bold u_0$.  Furthermore, for any $T_0\in (0, T)$ there exists a neighborhood $W_0\subset  H^s(\mathcal{G})\cap  \mathcal N_{0, Z}(s)$ of $\bold u_0$ such that the  data-to-solution  map
		$$
		\bold w_0 \in W_0 \mapsto    \bold w\in X_s(T_0)
		$$
		is  Lipschitz.
	\end{theorem}
	
The strategy of the proof of Theorem \ref{theorem1} is consider initially the case of only two edges, i.e. $| \bold E_{-}| = 1$ and $| \bold E_{+}| = 1$, and to use an auxiliary extended problem in the spirit of the paper of the second author \cite{Cava}  (see Section 3 below), where the solution will be obtained in the sense of the distributions and it is the restriction of a convenient extended problem for all line $\R$. It extended problem comes from every edge on the balanced star graph $\mathcal{G}$.  The extension of the problem in each edge on all $\mathbb{R}$ is non trivial since we need to recover the original boundary conditions \eqref{Zcondition}. The use of potentials for the linearized KdV equations on the positive and negative half-lines is fundamental here, since the exact formula of these potential permit us to define an appropriate integral equation that solves the problem on the distribution sense and satisfying \eqref{Zcondition}.

	\begin{remark}
		The natural regularity assumptions for the trace of the functions given by \eqref{Zcondition} are motivated by the	Kato smoothing effects obtained by Kenig, Ponce and Vega \cite{KPV}. 
	\end{remark}

\begin{remark}	The boundary conditions in \eqref{trace0} are based in a previous authors paper \cite{AC} concern to the instability properties of stationary solutions for KdV equation on balanced graphs with a profile determined by the domain $D(A_Z)$ in \eqref{domain8} (see Section 6 below).
\end{remark}

\begin{remark}
	The approach based on the Riemann-Liouville integrator operator used in the work \cite{Cav1 } does not work for the present context since the boundary condition \eqref{Zcondition} involves interaction between derivatives of different orders, preventing the construction of a integral formula by considering only a inversion of Riemann fractional integration. Thus, the approach considered here viewed as being more general in the sense of applications for more complicated boundary conditions.
\end{remark}

\begin{remark}
The result of Theorem \ref{theorem1} has the bound $s\geq 1$, which is necessary for ensuring the regularity of the second derivative of the trace function  $\partial_x^2 \bold u (0,t)$ belong unless on the space $L^{2}(0,T)$. As a consequence, the boundary condition \eqref{Zcondition} gains a more comprehensible meaning. Since regularity of this trace function is $H^{\frac{s-1}{3}}(0,T)$. It is possible to extend the result to a lower regularity assumption, more precisely, for $s\geq 0$, although the function $\partial_x^2 \bold u(0,t)$ would have a distributional sense.
\end{remark}

Next,  we establish that the mapping data-to-solution is not only Lipschitz but at least of class $C^2$ for $s=1$. This will be sufficient for  obtaining our nonlinear instability results for  specific stationary solutions of the KdV model on balanced graphs.

\begin{theorem}\label{dependence}
		Let $Z\neq 0$, and $\bold u_0 \in H^1(\mathcal{G})\cap  \mathcal N_{0, Z}(1)$. Then for  $T=T(\|\bold u_0\|_1)>0$ given by Theorem \ref{theorem1}, there is a neighborhood $V_0\subset  H^1(\mathcal{G})\cap  \mathcal N_{0, Z}(1)$ of $\bold u_0$ such that the  data-to-solution  map associated to problem in \eqref{kdv3}--\eqref{Zcondition}
		$$
		\bold w_0 \in V_0 \mapsto    \bold w\in X_s(T)
		$$
		is  of class $C^2$. 
\end{theorem}
 {The proof of Theorem \ref{dependence} is based in the Implicit Function Theorem and the estimates used in the proof of Theorem \ref{theorem1}. In section 5 below, for completeness in the exposition, we give an idea of the proof. We note that as the nonlinearity in the KdV model considered here is $C^\infty$, by using the strategy in \cite{Z} is possible to see that the  data-to-solution  map is also $C^\infty$.}
 \\

In the following we establish our second main result  in this manuscript which is associated to the nonlinear instability of some specific stationary solutions for the KdV model. We recall that  solutions of stationary type for the KdV  on a balanced graph $\mathcal G$ are determined in the form
$$
(u_{\bold e}(x,t))_{\bold e\in \bold E}=(\phi_{\bold e}(x))_{\bold e\in \bold E},\qquad\text{for all}\;\; t,
$$
where for $e \in \mathbf{E}_{-}$ the profile $\phi_{\mathbf{e}}:(-\infty, 0) \rightarrow \mathbb{R}$ satisfies $\phi_{\mathbf{e}}(-\infty)=0$, and for $\mathbf{e} \in \mathbf{E}_{+}, \phi_{\mathbf{e}}$ : $(0,+\infty) \rightarrow \mathbb{R}$ satisfies $\phi_{\mathbf{e}}(+\infty)=0$. Thus, from the the condition $\phi_{\mathbf{e}}(\pm\infty)=0$ we have that every component $\phi_{\bold e}$ satisfies {\it a priori} the following nonlinear elliptic equation
$$
\alpha_{\mathbf{e}} \frac{\mathrm{d}^2}{\mathrm{~d} x^2} \phi_{\mathbf{e}}(x)+\beta_{\mathbf{e}} \phi_{\mathbf{e}}(x)+\phi_{\mathbf{e}}^2(x)=0, \quad\quad \text{for all}\;  \mathbf{e} \in \mathbf{E}.
$$
For $\alpha_{\mathbf{e}}>0$ and $\beta_{\mathbf{e}}<0$, and for each $\mathbf{e} \in \mathbf{E}$, we can  obtain several families of profiles based on the classical soliton of  the KdV on  the full line,
\begin{equation}\label{soliton1}
	\phi_{\mathbf{e}}(x)=c\left(\alpha_{\mathbf{e}}, \beta_{\mathbf{e}}\right) \operatorname{sech}^2\left(d\left(\alpha_{\mathbf{e}}, \beta_{\mathbf{e}}\right) x+p_{\mathbf{e}}\right), \quad \mathbf{e} \in \mathbf{E},
\end{equation}
where the specific values of the shift $p_{\mathrm{e}}$ will depend on which other conditions are given for the profile $\phi_{\mathrm{e}}$ on the vertex of $\mathcal{G}$, $\nu=0$. In  \cite{AC} was studied the case of the stationary profile $(\phi_{\bold e})_{\bold e\in \bold E}$ belongs to the domain of the  family of   skew-self-adjoint extension $(A_Z, D(A_Z ))$  defined in \eqref{domain8} and  with the constants sequences  $(\alpha_ \bold e)_{\bold e\in \bold E}=(\alpha_{+})$ and  $(\beta_ \bold e)_{\bold e\in \bold E}=(\beta_{+}) $, with  $\alpha_{+}>0$ and $\beta_{+}<0$.   Thus, for $Z\neq 0$, $\omega\equiv -\frac{\beta_+}{\alpha_+}>\frac{Z^2}{4}$ and the half-soliton profiles $\phi_{\pm}$ defined by (see \eqref{solitary})
\begin{equation}\label{soli6}
	\phi_{+}(x)=-\frac{3\beta_{+}}{2} sech^2\Big(\frac{\sqrt{\omega}}{2}\; x-tanh^{-1}\Big(\frac{Z}{2\sqrt{\omega}}\Big)\Big),\;\;x>0
\end{equation}
and $\phi_{-}(x)\equiv \phi_{+}(-x)$ for $x<0$, we obtain for the  constants sequences of functions  
\begin{equation}\label{u+}
	u_{-}=(\phi_{-})_{\bold e\in \bold E_{-}},\; u_{+}=(\phi_{+})_{\bold e\in \bold E_{+}},
\end{equation}
that  $U_{Z}=(u_{-}, u_{+})$ represents one family of stationary profiles  for the KdV model on  a balanced graph and satisfying   $U_{Z}\in D(A_Z )$. For $Z>0$, $U_{Z}$ represents one family of stationary bump profiles (Figure 3)  and  for $Z<0$, $U_{Z}$ represents a family of stationary tail profiles (Figure 4).
\begin{figure}[h]
	\centering
	\begin{minipage}[b]{0.45\linewidth}\label{bump}
		\includegraphics[angle=0,scale=0.6]{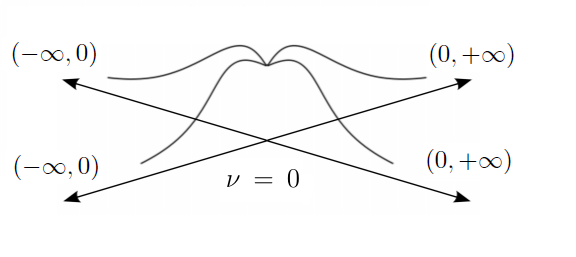}
		\caption{$U_{Z}$, $Z>0$, bump profiles}
	\end{minipage}
\end{figure}

\begin{figure}[h]
	\centering
	\begin{minipage}[b]{0.45\linewidth}\label{tail}
		\includegraphics[angle=0,scale=0.6]{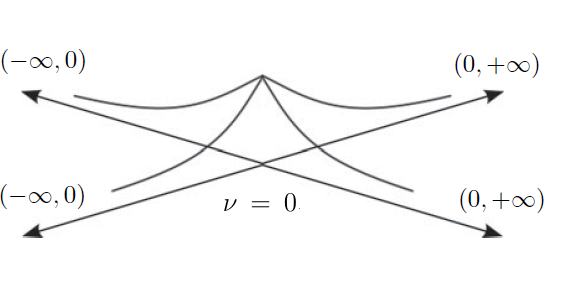}
		\caption{$U_{Z}$, $Z<0$, tail profiles}
	\end{minipage}
\end{figure}

Now, in  \cite{AC} was also established a criterion for the linear instability of stationary solutions for the KdV on arbitrary metric star graphs. This criterion was then  applied to the case of the profiles of type tail and bump $U_Z$ and so established that they are linearly unstable stationary solutions for the KdV model on balanced graphs (see Theorem 6.2 in \cite{AC}).

In the following we establish that the  linear instability result in Theorem \eqref{ins} of  \cite{AC} becomes a  nonlinear instability behavior. Thus we have the following definition.

\begin{definition}\label{nonstab}
	A stationary solution $\Phi =(\phi_{\bold e})_{\bold e\in \bold E} \in D(A_Z)$    is said to be \textit{nonlinearly unstable} in $X\equiv H^1(\mathcal G)$-norm by the flow of   the  Korteweg-de Vries model on $\mathcal G$ if there is $\epsilon>0$, such that for every $\delta>0$ there are a initial data $\bold u_0$ with $\|\Phi -\bold u_0\|_X<\delta$ and an instant $t_0=t_0(\bold u_0)$ such that $\|\bold u(t_0)-\Phi \|_X>\epsilon$, where $\bold u=\bold u(t)$ is the solution of the  Korteweg-de Vries problem in  \eqref{kdv3}   with initial data $\bold u(0)=\bold u_0$.
\end{definition}

Our nonlinear instability result is the following,

\begin{theorem}\label{Nins}  Let $Z\neq 0$,   $\alpha_{+}>0$,  $\beta_{+}<0$, and $-\frac{\beta_{+}}{\alpha_+}>\frac{Z^2}{4}$. Then for the profiles $\phi_{\pm}$  in \eqref{soli6} we have that  the family of tails and bumps profiles  $U_{Z}= (\phi_{\bold e})_{\bold e\in \bold E}$ with $\phi_{\bold e}= \phi_{-}$ for $\bold e\in \bold{E}_{-}$ and   $\phi_{\bold e}= \phi_{+}$ for $\bold e\in \bold{E}_{+}$, are nonlinearly unstable  for the  flow associated to the Korteweg-de Vries equation  on a balanced graph.
\end{theorem}

The strategy of proof for Theorem \ref{Nins} is to use the linear instability result  in Theorem \ref{ins} below and to apply the approach by Henry {\it {et al.}} \cite{HPW82}. The key point of this method is to use the fact that if the mapping data-solution associated to the  Korteweg-de Vries model   is of class $C^2$ on $H^1(\mathcal G)$, then we obtain the nonlinear instability from a result of linear instability (see Section 6 below).

 Lastly, the paper is organized as follows. In the Preliminaries (Section \ref{preliminaries}), we introduce some notations, the functional spaces used in this work, the free group associated with the Airy equation, the Duhamel inhomogeneous operator, and we enunciate some known estimates associated with functional spaces. A review concern potentials for a linearized KdV equation on the positive and negative half-lines is given on Section \ref{review}. Sections \ref{4} and \ref{5} is devoted to the proof of Theorems \ref{theorem1} and \ref{dependence}. Finally, Section \ref{6} is devoted to the proof of Theorem of nonlinear instability.

\section{Preliminaries}\label{preliminaries}

 \subsection{Sobolev spaces on $\mathbb R^{\pm}$ and $\mathcal G$}
 
 Initially, we denote the classical  Sobolev spaces on $\Omega\subset \mathbb R$ with the usual norm. We denote by  $\mathcal{G}$ the star graph constituted by $| \bold E_{-}|+| \bold E_{+}|$ half-lines ( $| \bold E_{-}|$-half-lines of the form  $(-\infty, 0)$ and $| \bold E_{+}|$-half-lines of the form   $(0, +\infty)$) attached to a common vertex $\nu=0$.   
 
 For $s\geq 0$ we say that $\phi \in H^s(\mathbb{R}^+)$ if exists $\tilde{\phi}\in H^s(\mathbb{R})$ such that 
 $\phi=\tilde{\phi}|_{\R+}$.  In this case we set $\|\phi\|_{H^s(\mathbb{R}^+)}:=\inf\limits_{\tilde{\phi}}\|\tilde{\phi}\|_{H^{s}(\mathbb{R})}$. For $s\geq 0$ define $$H_0^s(\mathbb{R}^+)=\Big\{\phi \in H^{s}(\mathbb{R}^+);\,\text{supp} (\phi) \subset[0,+\infty) \Big\}.$$ For $s<0$, define $H^s(\mathbb{R}^+)$ and $H_0^s(\mathbb{R}^+)$  as the dual space of $H_0^{-s}(\mathbb{R}^+)$ and  $H^{-s}(\mathbb{R}^+)$, respectively.

Also, define 
 $$C_0^{\infty}(\mathbb{R}^+)=\Big\{\phi\in C^{\infty}(\mathbb{R});\, \text{supp}(\phi) \subset [0,+\infty)\Big\}$$
 and $C_{0,c}^{\infty}(\mathbb{R}^+)$ as those members of $C_0^{\infty}(\mathbb{R}^+)$ with compact support. We recall that $C_{0,c}^{\infty}(\mathbb{R}^+)$ is dense in $H_0^s(\mathbb{R}^+)$ for all $s\in \mathbb{R}$. A definition for $H^s(\R^-)$ and $H_0^s(\R^-)$ can be given analogous to that for $H^s(\R^+)$ and $H_0^s(\R^+)$.
 
  On the graph $\mathcal G$ we define the classical spaces 
  \begin{equation*}
  L^p(\mathcal{G})=\bigoplus\limits_{\bold e\in  \bold E_{-} }L^p(-\infty, 0)  \oplus \bigoplus\limits_{\bold e\in  \bold E_{+} }L^p(0, +\infty), \quad \,p>1,
  \end{equation*}   
 and  for $s\geq0$
  \begin{equation*}  
 \quad H^s(\mathcal{G})=\bigoplus\limits_{\bold e\in  \bold E_{-} }H^s(-\infty, 0) \oplus \bigoplus\limits_{\bold e\in  \bold E_{+} } H^s(0, +\infty) 
 \end{equation*}   
with the natural norms. We emphasize that in these definitions we are not assuming any condition on the values of the functions at the vertex point $\nu=0$. For $u= (u_ \bold e)_{\bold e\in \bold E}, v= (v_ \bold e)_{\bold e\in \bold E}\in L^2(\mathcal G)$,  we also define
$$
\langle u, v\rangle\equiv \sum _{\bold e\in \bold E_-}\int_{-\infty}^0 u_ \bold e(x)\overline{v_ \bold e}(x)dx +\sum _{\bold e\in \bold E_+}\int_0^{\infty} u_ \bold e(x)\overline{v_ \bold e}(x)dx, 
$$
which turns $L^2(\mathcal G)$ into a  Hilbert space. Depending on the context we will use the following notations for different objects. By $||\cdot||$ we denote  the norm in $L^2(\mathbb{R})$ or in $L^2(\mathcal{G})$. By $||\cdot||_p$ we denote  the norm in $L^p(\mathbb{R})$ or in $L^p(\mathcal{G})$.

 \subsection{Bourgain spaces} Since our approach is based in to replace the original problem in \eqref{kdv3} for an convenient extended problem, in each edge, to all $\R$,  we will need to use appropriate function spaces defined in all $\R$. The principal idea is to construct an operator which is a contraction in adequate functional spaces. Thus, in this section we define and describe the principal properties of these spaces.

Initially, for a function $f$ in the Schwartz space $\mathcal S(\R)$ we denote the Fourier transform and the inverse transform of $f$, respectively by 
$$
\hat{f}(\xi)=\int_{\mathbb{R}} e^{-i \xi x} f(x) d x, \quad \mathcal{F}^{-1}[f](x)=f^{\vee}(x)=\frac{1}{2 \pi} \int_{\mathbb{R}} e^{i \xi x} f(\xi) d \xi.
$$
The classical Sobolev space based in $L^2(\R)$ is given by
$$
H^{s}(\mathbb{R})=\left\{f\in \mathcal S'(\R): \mathcal{F}^{-1}\left[(1+|\xi|)^{s} \hat{f}(\xi)\right] \in L_{2}(\mathbb{R})\right\},
$$
where $\mathcal S'(\R)$ denotes the spaces of tempered distributions.

Now we introduce special modified Bourgain spaces introduced by Faminskii. For $s \geq 0, a \in \mathbb{R}, b \in(0,1 / 2), \sigma \in(1 / 2,2 / 3)$ define the following Bourgain type spaces in the context of Faminskii works \cite{Faminskii} and \cite{Faminskii2}
$$
\begin{aligned} X^{s, b, \beta, \sigma}=\{& f(t, x) \in {S}^{\prime}\left(\mathbb{R}^{2}\right): \\ &\|f\|_{X^{s, b, \beta, \sigma}}=\left.\left(\iint_{\mathbb{R}^{2}}\left(1+|\xi|+|\tau|^{1 / 3}\right)^{2 s} \nu^{2}(\tau, \xi)|\hat{f}(\tau, \xi)|^{2} d \xi d \tau\right)^{1 / 2}<\infty\right\}, \end{aligned}
$$
where the function $\nu$ is given by
\begin{equation}
	\nu(\tau, \xi)=\nu_{\beta, b, \sigma}(\tau, \xi) \equiv\left(1+\left|\tau-\xi^{3}+\beta \xi\right|\right)^{b}+\chi_{[-1,1]}(\xi)(1+|\tau|)^{\sigma},
\end{equation}
\begin{equation}
	\begin{aligned} Y^{s,  b, \beta, \sigma}=\{& f(t, x) \in {S}^{\prime}\left(\mathbb{R}^{2}\right): \\&\quad\quad\|f\|_{Y^{s, b, \beta, \sigma}}=\left.\left(\iint_{\mathbb{R}^{2}}\left(1+|\xi|+|\tau|^{1 / 3}\right)^{2 s} \gamma^{2}(\tau, \xi)|\hat{f}(\tau, \xi)|^{2} d \xi d \tau\right)^{1 / 2}<\infty\right\}, \end{aligned}
\end{equation}
where the function $\gamma$ is defined as 
\begin{equation}
	\gamma(\tau, \xi)=\gamma_{ b,\beta,  \sigma}(\tau, \xi) \equiv \frac{1}{\left(1+\left|\tau-\xi^{3}+\beta \xi\right|\right)^{b}}+\frac{  \chi_{[-1,1]}(\xi)}{(1+|\tau|)^{1-\sigma}}.
\end{equation}

If $\Omega$ is a domain in $\mathbb{R}^{2},$ then define by $X ^{s, b, \beta, \sigma}(\Omega)$ and $Y^{s, b, \beta, \sigma}(\Omega)$ restrictions of $X^{s, b, \beta, \sigma}$
and $Y^{s, b, \beta, \sigma}$ on $\Omega,$ respectively, with natural restriction norms.

On the context of a balanced star graph $\mathcal G$ we consider de Bourgain   spaces $\bold {\mathcal  X}^{s,b,\beta,\sigma}(\mathcal G \times (0,T))$   by
	\begin{equation}\label{Bour}
		\bold {\mathcal X}^{s,b,\beta,\sigma}(\mathcal G \times (0,T))=\bigoplus\limits_{\bold e\in  \bold E_{-} }X^{s,b,\beta,\sigma}((-\infty,0)\times (0,T)) \oplus \bigoplus\limits_{\bold e\in  \bold E_{+} }X^{s,b,\beta,\sigma}((0,\infty)\times(0,T)).
\end{equation}

\subsection{Linear Group}  For convenience in the notation, we will consider the following linearized KdV equation,
\begin{equation}\label{lkdv}
\partial_tu+\partial_x^3u+\beta \partial_xu=0, \quad x\in \mathbb R
\end{equation}
for $\beta\in \mathbb R$. Then, the  linear group $S_{\beta}(t):=e^{-t(\partial_x^3+\beta\partial_x)}:S'(\R)\rightarrow S'(\R)$ associated with equation in \eqref{lkdv} is defined by
$$
e^{-t\left(\partial_x^3+\beta \partial_x\right)} \phi(x)=\left(e^{i t\left(\xi^3-\beta \xi\right)} \hat{\phi}(\xi)\right)^{\vee}(x),
$$
and will satisfy
$$
\left\{\begin{array}{lll}
	\left(\partial_t+\partial_x^3+\beta \partial_x\right) S_\beta(t) \phi(x)=0 & \text { for } & (x, t) \in \mathbb{R} \times \mathbb{R}, \\
	S_\beta(0) \phi(x)=\phi(x) & \text { for } & x \in \mathbb{R} .
\end{array}\right.
$$

In the rest of the paper, the function $\psi:\mathbb R\rightarrow \R$ will denote a cutt off regular function supported in the set $[-2,2]$ such that $\psi\equiv 1$ on the set $[-1,1]$. Also we denote $\psi_T(t)=\frac{1}{T}\psi(\frac{t}{T})$.

The next estimates were proven in \cite{Faminskii}.

\begin{lemma}\label{group} Let $s \geq 0$ and $0<b<1$. If $\phi \in H^s(\mathbb{R})$, then hold following estimates
	\begin{equation}
		\begin{split} 
			&\left\|\psi(t) S_\beta(t) \phi(x)\right\|_{C\left(\mathbb{R}_t ; H^s\left(\mathbb{R}_x\right)\right)}+\sum_{j=0}^2\left\|\psi(t) \partial_x^j S_\beta(t) \phi(x)\right\|_{C\left(\mathbb{R}_x ; H^{(s+1-j)/3}\left(\mathbb{R}_t\right)\right)}\\
			&\quad \quad+\left\|\psi(t) S_\beta(t) \phi(x)\right\|_{X^{s, b, \beta, \sigma}} \lesssim\|\phi\|_{H^s(\mathbb{R})}.
		\end{split}
	\end{equation}
\end{lemma}

\subsection{Duhamel  operator} The known inhomogeneous solution operator $\mathcal{K}$ associated with the KdV equation, for any $w \in L_{\mathrm{loc}}^1\left(\mathbb{R}, S^{\prime}(\mathbb{R})\right)$ is given by
\begin{equation}\label{Duhamel}
\mathcal{K} w(x, t)=\int_0^t e^{-\left(t-t^{\prime}\right)\left(\partial_x^3+\beta \partial_x\right)} w\left(x, t^{\prime}\right) \mathrm{d} t^{\prime},
\end{equation}
which satisfies in the sense of distributions
$$
\left\{\begin{array}{lll}
	\left(\partial_t+\partial_x^3+\beta \partial_x\right) \mathcal{K} w(x, t)=w(x, t) & \text { for } & (x, t) \in \mathbb{R} \times \mathbb{R}; \\
\mathcal	K w(x, 0)=0 & \text { for } & x \in \mathbb{R} .
\end{array}\right.
$$
Now, we summarize some useful estimates for the Duhamel inhomogeneous solution operators $\mathcal{K}$ that will be used later in the proof of the main results and its proof can be seen in \cite{Faminskii}.
\begin{lemma}\label{duhamel} For all $s \in \mathbb{R}$, we have the following estimate
	$$
	\begin{aligned}
		\|\psi(t) \mathcal{K} w(x, t)&\|_{C\left(\mathbb{R}_t ; H^s\left(\mathbb{R}_x\right)\right)}+\sum_{j=0}^2\left\|\psi(t) \partial_x^j \mathcal{K} w(x, t)\right\|_{C\left(\mathbb{R}_x ; H^{(s+1) / 3}\left(\mathbb{R}_t\right)\right)} \\
		&\quad+\|\psi(t) \mathcal{K} w(x, t)\|_{X^{s, b, \beta, \sigma}} \lesssim\|w\|_{Y{s, b, \beta, \sigma}} .
	\end{aligned}
	$$
\end{lemma}
The following lemma state a estimate concern the $X^{s,b,\beta,\sigma}$ spaces.

\begin{lemma}\label{psiT} For $f \in X^{s, b, \beta, \sigma}, T>0$ holds
	$$
	\left\|\psi_T(t) f(t, x)\right\|_{X^{s, b, \beta, \sigma}} \leq c(s, b, \beta, \sigma) T^{1 / 2-\sigma-s / 3}\|f(t, x)\|_{X^{s, b, \beta, \sigma}} .
	$$
\end{lemma}

\begin{remark}\label{power}
In our context, we will work with $\sigma > \frac{1}{2}$. Therefore, the exponent $r := \frac{1}{2} - \sigma - \frac{s}{3}$ can become negative depending on the value of $s$. In the case where this exponent is negative, we can modify the time to $\tilde{T} = T^{-\theta}$ for any $\theta > 0$. Thus, in Lemma \eqref{psiT}, we obtain the new term $\tilde{T}^{\frac{1}{2} - \sigma - \frac{s}{3}} = T^{-\theta\left(\frac{1}{2} - \sigma - \frac{s}{3}\right)}$. We use this argument occasionally, as in the fixed-point argument, we require a positive power of time.
\end{remark}

Next, nonlinear estimates, in the context of the $\mathrm{KdV}$ equation equation, for $b<\frac{1}{2}$, were derived by Faminskii in \cite{Faminskii}.

\begin{lemma}\label{bilinear}Let $b \in[7 / 16,1 / 2), \quad \sigma \in(1 / 2,2 / 3), \quad s \geq 0$, we have
	$$
	\left\|\psi(t) \partial_x\left(v_1 v_2\right)\right\|_{Y^{s, b, \beta, \sigma}} \lesssim\left\|v_1\right\|_{X^{s, b, \beta, \sigma}}\left\|v_2\right\|_{X^{s, b, \beta, \sigma}}.$$
\end{lemma}

\section{A review concern potentials for a linearized KdV equation on the positive and negative half-lines.}\label{review}

An important point in this paper is the use of the corresponding linear problems associated with the KdV equation on the positive and negative half-lines. More precisely we 
consider the following IBVPs associated with the linearized $\mathrm{Kd} \mathrm{V}$ on the positive half-line
\begin{equation}\label{kdvr}
	\left\{\begin{array}{l}
		\partial_t v+\partial_x^3 v+\beta \partial_x v=0,\quad x \in(0,+\infty), t \in(0, T); \\
		v(x, 0)=v_0(x), \quad\quad\quad\;\; x \in(0,+\infty); \\
		v(0, t)=f(t), \quad\quad\quad\;\;\;\; t\in(0, T)
	\end{array}\right.
\end{equation}
and the negative half-line
\begin{equation}\label{kdvl}
	\left\{\begin{array}{l}
		\partial_t w+\partial_x^3 w+\beta \partial_x w=0, \quad x \in(-\infty, 0), t \in(0, T);\\
		w(x, 0)=w_0(x),\quad\quad\quad\;\;\; x \in(-\infty, 0);\\
		w(0, t)=g(t), \quad\quad\quad\;\;\;\;\;\; t \in(0, T);\\
		\partial_x w(0, t)=h(t), \quad\quad\quad\;\;\; t \in(0, T).
	\end{array}\right.
\end{equation}

Here, we assume that $v_0\in H^s(\mathbb R^+)$, $w_0\in H^s(\mathbb R^-)$, $f,g\in H^{\frac{s+1}{3}}(0,T)$ and $h\in H^{\frac s 3}(0,T)$.

The presence of one boundary condition on the right half-line problem \eqref{kdvr} versus two boundary conditions on the left half-line problem \eqref{kdvl} can be motivated by integral
identities on smooth solutions to the linear KdV equation, for details the reader can see \cite{Holmer}.

For a complete survey about the problems \eqref{kdvr} and \eqref{kdvl} the reader can see \cite{Cav2}. Also, results of orbital stability and asymptotic stability can be seen in \cite{CM2} and \cite{CM3}.

In this section, we describe explicitly the solutions for the linearized $\mathrm{KdV}$ equation and its properties based in some potentials. For the first time, such potentials were introduced in the paper Cattabriga \cite{Cat}. After, Faminskii \cite{Faminskii} used this potential solutions to solve the nonlinear problems for the KdV on the positive and the negative half-lines. The interesting point in theses papers is that the proof of the global existence of solutions for the KdV equations on the half-lines by assuming more natural conditions for the boundary functions.

In the rest of this section we assume $\beta<0$ in \eqref{kdvr} and \eqref{kdvl}. In order to get the exact formulas for the IBVP \eqref{kdvr} and \eqref{kdvl}, by following the Faminskii works \cite{Faminskii} and \cite{Faminskii2}, we need to study the  following algebraic equation
\begin{equation}\label{algebraic}
	z^3+\beta z+\varepsilon+i \tau=0.
\end{equation}
If $\varepsilon>0$, for any fixed $\tau$, then this equation has one root $z_0$ such that $\operatorname{Re} z_0<0$, and two roots $z_1$ and $z_2$ such that $\operatorname{Re} z_1>0$ and $\operatorname{Re} z_2>0$. Let
$$
r_j(\tau)=r_j(\tau ; \beta)=\lim _{\varepsilon \rightarrow+0} z_j(\varepsilon+i \tau), \quad j=0,1,2 .
$$
The Faminskii papers \cite{Faminskii} and \cite{Faminskii2} gives the following exact formula
\begin{equation}\label{root}
	\begin{aligned}
		&r_0(\tau)=-p(\tau)+i q(\tau), \\
		&r_1(\tau)=p(\tau)+i q(\tau), \\
		&r_2(\tau)=i k_\beta(\tau),
	\end{aligned}
\end{equation}
where $p$ and $q$ are real valued functions and $k_{\beta}$ is the inverse of function $\phi_{\beta}=\xi^3-\beta \xi$.

It is possible to show the following estimate for $j=0,1,2$,

\begin{equation}\label{est1} \quad\left|r_j(\tau)\right| \leq c\left(|\tau|^{1 / 3}+|\beta|^{1 / 2}\right).
\end{equation}
For a certain constant $c>0$
\begin{equation}\label{estp} p(\tau), q(\tau) \geq c\left(|\tau|^{1 / 3}+|\beta|^{1 / 2}\right) \quad \forall \tau \in \mathbb{R}.
\end{equation}

Now, we give the exact formulas for the IBVPs \eqref{kdvr} and \eqref{kdvl}. Initially let $\tilde h$, $\tilde f$ and $\tilde g$ extensions of the functions $h,$ $f$ and $g$, respectively, on all real line $\R$, satisfying
$$
\|\tilde f\|_{H^{\frac{s+1}{3}}(\R^+)}\lesssim \|f\|_{H^{\frac{s+1}{3}}(0,T)},\ \|\tilde g\|_{H^{\frac{s+1}{3}}(\R^+)}\lesssim \|g\|_{H^{\frac{s+1}{3}}(0,T)}\ \text{and}\ \|\tilde h\|_{H^{\frac{s}{3}}(\R^+)}\lesssim \|h\|_{H^s(0,T)}.
$$

For $\tilde{h} \in S^{\prime}(\mathbb{R})$, $t\in \R$ and $x>0$ we define
\begin{equation}\label{R}
\mathcal R_\beta(t, x ; \tilde{h})=\mathcal{F}_t^{-1}\left[e^{r_0(\tau) x} \hat{\tilde{h}}(\tau)\right](t) .
\end{equation}
For $\tilde{f}, \tilde{g} \in \mathcal{S}^{\prime}(\mathbb{R})$ define for $t\in \R$ and $x \leq 0$
\begin{equation}\label{L1}
\mathcal L_\beta^1(t, x ; \tilde{f}) \equiv \mathcal{F}_t^{-1}\left[\frac{r_1(\tau) e^{r_2(\tau) x}-r_2(\tau) e^{r_1(\tau) x}}{r_1(\tau)-r_2(\tau)} \hat{\tilde{f}}(\tau)\right](t)
\end{equation}
and
\begin{equation}\label{L2}
\mathcal L_\beta^2(t, x ; \tilde g) \equiv \mathcal{F}_t^{-1}\left[\frac{e^{r_1(\tau) x}-e^{r_2(\tau) x}}{r_1(\tau)-r_2(\tau)} \hat{\tilde{g}}(\tau)\right](t) .
\end{equation}
As the consequence of \eqref{R}, \eqref{L1} and \eqref{L2} we have the following essential traces for the functions $R_\beta \tilde{h}, L_\beta^1 \tilde{f}, L_\beta^1 \tilde{g}$ and its spatial derivatives,
\begin{equation}\label{traces}
\left\{\begin{array}{l}
 \mathcal R_{\beta} \tilde{h}(0, t)=\tilde{h}(t), \mathcal L_{\beta}^1 \tilde{f}(0, t)=\tilde{f}(t),\mathcal L_{\beta}^2 \tilde{g}(0, t),\ t\in\R; \\
\partial_x(\mathcal R_{\beta} \tilde{h})(0, t)=\mathcal{F}_t^{-1}\left[r_0(\tau) \hat{h}\right](t),\ t\in\R; \\
\partial_x\left(\mathcal L_{\beta}^1 \tilde{f}\right)(0, t)=0,\partial_x\left(\mathcal L^2_{\beta} \tilde{g}\right)(0, t)=\tilde{g}(t),\ t\in\R; \\
\partial_{x}^2(\mathcal R_{\beta} \tilde{h})(0, t)=\mathcal{F}_t^{-1}\left(r_0^2(\tau) \tilde{h}\right)(t),\ t\in\R; \\
\partial_{x}^2\left(\mathcal L_{\beta}^1 \tilde{f}\right)(0, t)=-\mathcal{F}^{-1}\left(r_1(\tau) r_2(\tau) \hat{\tilde{f}}\right)(t),\ t\in\R; \\
\partial_{x}^2\left(\mathcal L^2_{\beta} \tilde{g}\right)(0, t)=\mathcal{F}_t^{-1}\left(\left(r_1(\tau)+r_2(\tau)\right) \hat{\tilde{g}}\right)(t),\ t\in\R;.
\end{array}\right.
\end{equation}

\begin{remark}
The values of the traces given in \eqref{traces} will be essential in the construction of an integral formula equivalent to the problem \eqref{kdv3} 	satisfying the boundary condition \eqref{Zcondition}.
\end{remark}

\begin{lemma}
For any $h\in H^{\frac{s+1}{3}}(\R^+)$ and $s\geq 0$ we have that $\mathcal{R}h\in C(\R^+;H^{s}(\mathbb R^+))$. Moreover holds the following inequality
\begin{equation}
\|\mathcal{R}h(\cdot,t)\|_{C(\R^+;H^s(\R+))}\leq c(\beta,s) \|h\|_{H^{\frac{s+1}{3}}(\R^+)}
\end{equation}
\end{lemma}
\begin{proof}
This result was obtained in \cite{Faminskii}. We only provide a sketch of the proof, as the argument of the proof will be used to make appropriate extensions of the function $\mathcal{R}h$ on the entire plane $\mathbb{R}^2.$
The following calculations will be usefull in this work.
Without loss of generality, we assume $h \in C_{0}^{\infty}(\mathbb  R^+)$. Let $n$ a nonegative integer number. By using \eqref{R} for $x>0, t \geq 0$
$$
D_x^n \mathcal R_{\beta}(t, x ; \mu)=\frac{1}{2 \pi} \int_{\mathbb{R}} e^{i \lambda t}(r(\lambda))^n e^{r(\lambda) x} \hat{h}(\lambda) d \lambda ,
$$
where abusing the notation we consider $h$ the extension for zero on the negative half-line. 
Now, we consider $\beta>0$, then
\begin{equation*}
\begin{split}
D_x^n \mathcal R_{\beta}(t, x ; h)&=\frac{1}{2 \pi} \int_{|\lambda|<2(\beta/ 3)^{3 / 2}}e^{i \lambda t}(r(\lambda))^n e^{r(\lambda) x} \hat{h}(\lambda)  d \lambda+\frac{1}{2 \pi} \int_{|\lambda|>2(\beta / 3)^{3 / 2}} e^{i \lambda t}(r(\lambda))^n e^{r(\lambda) x} \hat{h}(\lambda)  d \lambda \\
&\equiv I_1+I_2 .
\end{split}
\end{equation*}
Consider first $I_1$ (in this case $r(\lambda)=i q(\lambda)$ ). Changing variables $\xi=q(\lambda)$ we find
$$
I_1=\frac{1}{2 \pi} \int_{\mathbb{R}} e^{i\left(\xi^3-\beta \xi\right) t} e^{i \xi x} \hat{v}_n(\xi) d \xi=S_{\beta}\left(t, x ; v_n\right),
$$
where
\begin{equation}\label{vn}
v_n(x) \equiv-\mathcal{F}_x^{-1}\left[\chi \sqrt{\beta / 3}(\xi)\left(r\left(\xi^3-\beta \xi\right)\right)^n\left(3 \xi^2-\beta\right) \hat{h}\left(\xi^3-\beta \xi\right)\right](x).
\end{equation}
and according to the Parseval equality
$$
\left\|I_1(t, \cdot)\right\|_{L_{2,+}} \leq\left\|I_1(t, \cdot)\right\|_{L_2} \leq\left\|\hat{v}_n\right\|_{L_2} \leq c(\beta, n)\|\mu\|_{L_{2,+}} .
$$
For $I_2$ changing variable $\xi=\kappa_{\beta}(\lambda)$ we find, that
$$
I_2=\frac{1}{2 \pi} \int_{|\xi|>2 \sqrt{\beta / 3}} e^{i\left(\xi^3-\beta \xi\right) t}\left(r\left(\xi^3-\beta \xi\right)\right)^n e^{r\left(\xi^3-\beta \xi\right) x} \hat{\mu}\left(\xi^3-\beta \xi\right)\left(3 \xi^2-\beta\right) d \xi .
$$
By using the following inequality from \cite{bona2002} if certain continuous function $\gamma(\lambda)$ satisfies an inequality $\Re \gamma(\lambda) \leq-\varepsilon|\lambda|$ for some $\varepsilon>0$ and all $\lambda \in \mathbb{R}$, then
$$
\left\|\int_{\mathbb{R}} e^{\gamma(\lambda) x} f(\lambda) d \lambda\right\|_{L^{2}(\R^+)} \leq c(\varepsilon)\|f\|_{L_2} .
$$
By following \cite{Faminskii} we have that
$$
\begin{aligned}
\left\|I_2(t, \cdot)\right\|_{L^{2}(\mathbb R^+)} & \leq\left\|\left(1-\chi_{2 \sqrt{\beta / 3}}(\xi)\right)\left(r\left(\xi^3-\beta \xi\right)\right)^n \hat{h}\left(\xi^3-\beta \xi\right)\left(3 \xi^2-\beta\right)\right\|_{L_2} \\
& \leq c_1\|h\|_{H_{0}^{(n+1) / 3}(\mathbb R^+)}
\end{aligned}
$$
Combining the obtained estimates we derive (3.29) for $s=n$ and $\beta>0$. Finally, the case $\beta<0$ the partition of the integral is unnecessary.
\end{proof}

The following lemma obtained in \cite{Faminskii} and \cite{Faminskii2} gives the solutions for the linearized problems \eqref{kdvr} and \eqref{kdvl}.

\begin{lemma}\label{slr} Let $f \in H^{\frac{s+1}{3}}\left(\mathbb{R}^{+}\right), g \in H^{\frac{s+1}{3}}\left(\mathbb{R}^{+}\right)$and $h \in H^{\frac{s+1}{3}}\left(\mathbb{R}^{+}\right)$. 
	\begin{itemize}
\item[(i)] For any extension $\tilde{h} \in H^{\frac{s+1}{3}}(\mathbb{R})$ of $h$ on the entire line $\mathbb{R}$, the function $\left.\mathcal R_{\beta} \tilde{h }\right|_{(0,\infty)\times (0, T)}$ is the generalized solution of IBVP \eqref{kdvr}.
\item[(ii)] For any extensions $\tilde{f},\ \tilde g \in H^{\frac{s+1}{3}}(\mathbb{R})$ of  the functions $f$ and $g$ on the entire line the function $\left.\left(\mathcal L_{\beta}^1 \tilde f+\mathcal L^2_{\beta} \tilde g\right)\right|_{(-\infty,0)\times(0, T)}$ is  the generalized solution of the IBVP \eqref{kdvl}.
\end{itemize}
\end{lemma}
\begin{remark}\label{uni}
By using the result of uniqueness of Faminskii  \cite{Faminskii,Faminskii2} we have that the restriction on the set $[0,T]$ of the formulas obtained in Lemma \ref{slr}  for the solutions of \eqref{kdvr} and \eqref{kdvl} do not depend of the extensions $\tilde f$, $\tilde g$ and $\tilde h$.
	\end{remark}

\begin{remark}\label{ext1}
{(Extensions on all $\mathbb R^2$)}  Now, to work within the context of Bourgain spaces, we need to extend the formulas \eqref{R} for $x<0$ and $t<0$, as well as extend the formula \eqref{L1} for $x>0$ and $t<0$. This extension is performed following the methods outlined in the works by \cite{Faminskii} and \cite{Faminskii2}. First, to deal with \eqref{R}, we define the following function,
$$
\sigma(x)=\sigma(x ; \lambda) \equiv\left\{\begin{array}{l}
	e^{p(\lambda) x}, \quad x \geq 0, \\
	c_0 e^{-p(\lambda) x}+c_1 e^{-p(\lambda) x / 2}+\cdots+c_n e^{-p(\lambda) x / 2^n}, \quad x \leq 0,
\end{array}\right.
$$
where the coefficients $c_0, c_1, \ldots c_n$ are chosen such, that
$$\left\{\begin{array}{l}
	c_0+c_1+\cdots+c_n=1, \\
	(-1)\left(c_0+\frac{1}{2} c_1+\cdots+\frac{1}{2^n} c_n\right)=1, \\
	\left.\cdots \cdots \cdots \cdots \cdots \cdots \cdots \cdots\right. \\
	(-1)^n\left(c_0+\frac{1}{2^n} c_1+\cdots \cdots+\frac{1}{2^{n^2}} c_n\right)=1.
\end{array}\right.$$
Note that the system above has a solution, and with these coefficients, the function $\rho \in H^{n+1}(\mathbb R)$. 
Thus, we put for any $(x,t) \in \mathbb{R}^2$ in the case $\beta \leq 0$
$$
R_{\beta}(t, x ; \tilde h) \equiv \mathcal{F}_t^{-1}\left[\sigma(x ; \lambda) e^{i q(\lambda) x} \hat{\tilde h}(\lambda)\right](t),
$$
and in the case $\beta>0$
$$
\begin{aligned}
	R_{\beta}(t, x ; \mu) & \equiv S_{\beta}\left(t, x ; v_0\right)+\mathcal{F}_t^{-1}\left[\sigma(x ; \lambda) e^{i q(\lambda) x} \hat{\mu}(\lambda)\left(1-\chi_{2(\beta / 3)^{3 / 2}}(\lambda)\right)\right](t) \\
	& \equiv S_{\beta}\left(t, x ; v_0\right)+J_{0 \beta}(t, x ; \mu),
\end{aligned}$$ 
where $v_0$ is given by \eqref{vn}.
An extension of the functions $\mathcal L_{\beta}^1\tilde f$  and $\mathcal{L}_{\beta}^2 \tilde g$ can be done in a similar way, for more details the reader can see 	\cite{Faminskii2}.
\end{remark}

\begin{remark}
Given any functions $f \in H^{\frac{s+1}{3}}\left(\mathbb{R}^{+}\right), g \in H^{\frac{s+1}{3}}\left(\mathbb{R}^{+}\right)$ and $h \in H^{\frac{s+1}{3}}\left(\mathbb{R}^{+}\right)$, in the rest of the paper we can consider the functions $\mathcal R h$, defined on the set $(x,t)\in (\R^+)^2$, as well the functions $\mathcal L_{\beta}^1 f +\mathcal L_{\beta}^2 g$ on the set $(x,t)\in \R^-\times \R^+$ . Sometimes, we will consider its appropriate extensions denoted by $\mathcal R \tilde h$ and  $\mathcal L_{\beta}^1 \tilde f +\mathcal L_{\beta}^2 g$, respectively,  in all plane $\R^2$ in the spirit of Remark \ref{ext1}.
\end{remark}

The following lemma states the fundamental estimates for the potential solutions and was obtained by Faminskii in \cite{Faminskii} and \cite{Faminskii2}. 

\begin{lemma}
 Let $s \geq 0$. For any functions $f, h \in H^{\frac{s+1}{3}}(\mathbb{R}^+)$ and $g\in H^{\frac{s}{3}}(\mathbb R^+)$ the following estimates are ensured:

\begin{itemize}
\item[(a)]
\begin{equation*}
\begin{split}
	&\|\psi(t)\mathcal R_{\beta} \tilde h(x, t)\|_{C\left(\mathbb{R}_t ; H^s\left(\mathbb{R}_x\right)\right)}+\left\|\psi(t)\partial_x^n(\mathcal R_{\beta} \tilde h(x, t))\right\|_{C \left(\mathbb{R}_x; H^{\frac{s+1-n}{3}}\left(\mathbb{R}_t\right)\right)}\\
	&+\|\psi(t) \mathcal R_{\beta} \tilde h(x, t)\|_{X_{s, b, \beta, \sigma}} \lesssim\| h\|_{H^{\frac{s+1}{3}}\left(\mathbb{R}^{+}\right)},\quad (n=1,2, \ldots)
\end{split}
\end{equation*}

\item[(b)]
\begin{equation*}
	\begin{split}
&\left\|\psi(t)\mathcal L_{\beta}^1 \tilde f(x, t)\right\|_{C\left(\mathbb{R}_t ; H^s\left(\mathbb{R}_x\right)\right)}+\left\|\psi(t)\partial_x^n\left(\mathcal L_{\beta}^1 \tilde f(x, t)\right)\right\|_{C\left(\mathbb{R}_x ; H^{\frac{s+1-n}{3}}\left(\mathbb{R}_t\right)\right)} \\
&\quad+\left\|\psi(t) \mathcal L_{\beta}^1 \tilde f(x, t)\right\|_{X_{s, b, \beta, \sigma}} \lesssim \| f \|_{H^{\frac{s+1}{3}}\left(\mathbb{R}^{+}\right)},\quad(n=1,2, \ldots)
\end{split}
\end{equation*}
\item[(c)]
\begin{equation*}
	\begin{split}
&\left\|(\psi(t))\mathcal L_{\beta}^2 \tilde g(x, t)\right\|_{C\left(\mathbb{R}_t ; H^s\left(\mathbb{R}_x\right)\right)}+\left\|(\psi(t))\partial_x^n\left(\mathcal L_{\beta}^2 \tilde g(x, t)\right)\right\|_{C\left(\mathbb{R}_x; H^{\frac{s+1-n}{3}}\left(\mathbb{R}_t\right)\right)} \\
&\quad+\left\|\mathcal L_{\beta}^2 \tilde  g(x, t)\right\|_{X_{s, b, \beta, \sigma}} \lesssim\| g\|_{H^{s/3}\left(\mathbb{R}^{+}\right)}, \qquad (n=1,2, \ldots) .
\end{split}
\end{equation*}
\end{itemize}
\end{lemma}

\section{An integral equation that solves the problem on a balanced graph}\label{4}

Initially, we point out  that it suffices for our proof of Theorem \ref{theorem1}  to consider the case of only two edges, i.e. $|\bold{E}_{-}|=1$ and $|\bold{E}_{+}|=1$, because  on the general case of  a balanced graph $\mathcal G$  the conditions \eqref{Zcondition} in the vectorial notation takes the following form (for $\bold{u}=(u_{1},...,u_{n}, v_{1},..., v_{n})$)
\begin{equation}\label{Z*}
	\begin{cases}
		(u_{1}(0-,t), u_2(0-,t),...,u_{n}(0-,t))=(v_{1}(0+,t),...,v_{n}(0+,t),\ \ \text{in\ the\ sense\ of } H^{\frac{s+1}{3}}(0,T);\\ v'_{k}(0+,t)- u'_{k}(0-,t)=Zu_k(0-,t),\ \ (k=1,...,n)\  \text{in\ the\ sense\ of } H^{\frac{s}{3}}(0,T);\\  \frac{Z^2}{2}u_{k}(0-,t)+Zu'_{k}(0-,t)=v''_{k}(0+,t)-u''_{k}(0-,t),\ \ (k=1,...,n)\ \text{in\ the\ sense\ of } H^{\frac{s-1}{3}}(0,T).
	\end{cases}
\end{equation}

\medskip

   Now, fix $s \geq 1, b \in[7 / 16,1 / 2)$ and $\alpha \in(1 / 2,2 / 3)$. Our strategy is to extend the problem to all lines $\mathbb{R}$ along each edge, aiming to derive an integral equation that solves the extended problem. This solution should be such that when restricted to each edge, it solves the original problem \eqref{kdv3}-\eqref{Zcondition} in the sense of distributions.  In this context, our solution will be denoted as $\bold u=(u, v)=\left(\left.\tilde{u}\right|_{\mathbb{R}^{+}\times (0,T)},\left.\tilde{v}\right|_{\mathbb{R}^{-}\times(0,T)}\right)$, where $\tilde{u}$ and  $\tilde{v}$ are solutions of the corresponding extended problem for all $\mathbb{R}$ along each edge, in the sense of distributions.

To do this, we start by taking extensions $\widetilde{u}_0$ and $\widetilde{v}_0$ of the initial data $u_0$ and $v_0$ satisfying
$$
\left\|\widetilde{u}_0\right\|_{H^s(\mathbb{R})} \leq c\left\|u_0\right\|_{H^s\left(\mathbb{R}^{+}\right)} \text {and }\left\|\widetilde{v}_0\right\|_{H^s(\mathbb{R})} \leq c\left\|v_0\right\|_{H^s\left(\mathbb{R}^{-}\right)} \text {. }
$$

 We will find solutions for the Cauchy problem \eqref{kdv3}-\eqref{Zcondition} as the restriction of the functions satisfying  following integral equations
\begin{equation}\label{u}
\tilde u(x,t)=\psi(t)S_{\beta}(t)\tilde u_0+\psi(t)\mathcal{K}(\psi_T^ 2(t)\tilde u\tilde u_x)+\psi(t)\mathcal R_{\beta}h(x,t)
\end{equation}
and 
\begin{equation}\label{v}
\tilde v(x,t)=\psi(t)S_{\beta}(t)\tilde v_0+\psi(t)\mathcal{K}(\psi_T^ 2(t)\tilde v\tilde v_x)+\psi(t)\mathcal L_{\beta}^1f(x,t)+\psi(t)\mathcal L_{\beta}^2g(x,t),
\end{equation}
where the functions $f$, $g$ and $h$ depend of the initial conditions $\widetilde{u}_0$ and $\widetilde{v}_0$,  as well as the unknown functions $\tilde u$ and $\tilde v$.  These functions will be chosen in a way that satisfies the boundary conditions \eqref{Zcondition}. Here, $\mathcal{K}$ represents the classical Duhamel operator defined in \eqref{Duhamel}. With suitable choices for $f$, $g$, and $h$, we will observe that equations \eqref{u} and \eqref{v} become integral equations for the Cauchy problem \eqref{kdv3}-\eqref{Zcondition}.

In order to simplify the notation we denote 

\begin{equation}\label{F1}{F_1(x,t)}:=F_1(\tilde u_0,\tilde u,x,t)=\psi(t)S_{\beta}(t)\tilde u_0+\psi(t)\mathcal{K}(\psi_T^ 2(t)\tilde u\tilde u_x)\end{equation} and 
\begin{equation}\label{F2}{F_2(x,t)}:=F_2(\tilde v_0, \tilde v,x,t)=\psi(t)S_{\beta}(t)\tilde v_0+\psi(t)\mathcal{K}(\psi_T^ 2(t)\tilde v \tilde v_x).
\end{equation}

 Now, we start the process of choices of unknown functions by making use of the boundary conditions \eqref{Zcondition}. 
 
 Initially, by using \eqref{u}, \eqref{v} and  the boundary conditions \eqref{Zcondition} we impose that the functions $f$, $g$ and $h$ must satisfy the following relations
\begin{equation}\label{trace1}
\tilde u(0,t)=\tilde v(0,t)\implies h(t)+F_1(0,t)=F_2(0,t)+f(t)\implies \hat h-\hat f=\hat F_2(0,\tau)-\hat F_1(0,\tau),\end{equation}

\begin{equation}\label{trace2}
\begin{split}& \tilde u'(0,t)-\tilde v'(0,t)=Z\tilde \tilde v(0,t)\implies r_0\hat h+ \partial_x \hat{F}_{1}(0,t)-\hat g- \partial_x \hat F_{2}(0,t)=Z(\hat f+\hat F_2)\\
&\implies r_0\hat h-\hat g-Z\hat f=- \partial_x\hat{F}_{2}(0,\tau)+Z\hat F_2+ \partial_x\hat F_{2}(0,\tau)
\end{split}\end{equation}
and
\begin{equation}\label{trace3}
\begin{split}
&\tilde u''(0,t)-\tilde v''(0,t)=\frac{Z^2}{2}v(0,t)+Zv'(0,t)\\
&\implies r_0^2(\tau)\hat h+\partial_x^2\hat F_{1}+r_1(\tau)r_2(\tau)\hat f
-(r_1(\tau)+r_2(\tau))\hat g-\partial_{x}^2\hat{F}_{2}=\frac{Z^2}{2}(\hat{f}+\hat{F}_2)+Z(\hat g+\partial_x\hat F_{2}).\end{split}
\end{equation}

Here $\hat{}$ denotes the Fourier transform on the variable $t$.

The expressions  \eqref{trace1}, \eqref{trace2} and \eqref{trace3} are written in the context of matrices in the following form:
\begin{equation}\label{tec0}
\left[ \begin{array}{ccc}
	1 & -1  & 0 \\ 
	r_0(\tau) & -Z & -1\\
	r_0^2(\tau) & r_1(\tau)r_2(\tau)-\frac{Z^2}{2}  & -(r_1(\tau)+r_2(\tau)+Z)
\end{array} \right] \left[ \begin{array}{c}
\hat h  \\ 
\hat f\\
\hat g 
\end{array} \right]=\left[ \begin{array}{c}
\hat{F}_2 -\hat{F}_1 \\ 
-\partial_x\hat{F}_{1}+\partial_{x}\hat{F}_{x}+Z\hat{F}_2\\
-\partial_{x}^2\hat{F}_{1}+\partial_{x}^2\hat{F}_{2}+\frac{Z^2}{2}\hat{F}_2+Z\partial_{x}\hat{F}_{2}
\end{array} \right].
\end{equation}
We denote this in the reducing form
\begin{equation}\label{matrix}
M\bold{f}:=M(r_0(\tau),r_1(\tau),r_2(\tau),Z)\bold f=\hat{\bold F},
\end{equation}
where $M$ denotes the first matrix in \eqref{tec0}, $\bold{f}$ denotes de column vector with coordinates $\hat h$, $\hat f$ and $\hat g$, and $\hat{\bold F}$ the column vector with entries $\hat{F}_2 -\hat{F}_1$, 
$-\partial_x\hat{F}_{1}+\partial_x\hat{F}_{2}+z\hat{F}_2$ and 
$-\partial_x^2\hat{F}_{1}+\partial_x^2\hat{F}_{2}+\frac{Z^2}{2}\hat{F}_2+Z\partial_x\hat{F}_{2}$.

The determinant of $M$ is given by
\begin{equation}
\det M= \frac{Z^2}{2}+(r_1(\tau)+r_2(\tau)-r_0(\tau))+r_1(\tau)r_2(\tau)-r_0(\tau)r_1(\tau)-r_0(\tau)r_2(\tau).
\end{equation}

Now, denotes $k_{\beta}$ the inverse of function $\phi_{\beta}(\xi)=\xi^3-\beta \xi$. By using \eqref{root} we have
\begin{equation}
\begin{split}
&r_1+r_2-r_0=2p+ik_{\beta}\\
&r_1r_2=-qk_{\beta}+ik_{\beta}p\\
&-r_0r_1=q^2+p^2,\\
&-r_0r_2=qk_{\beta}-pk_{\beta}i.
\end{split}
\end{equation}
It follows that 
\begin{equation}
r_1(\tau)r_2(\tau)-r_0(\tau)r_1(\tau)-r_0(\tau)r_2(\tau)=q^2(\tau)+p^2(\tau)
\end{equation}
and
\begin{equation}
\det M(r_0(\tau),r_1(\tau),r_2(\tau),Z)= \frac{Z^2}{2}+ q^2(\tau)+p^2(\tau)+
2p(\tau)+ik_{\beta}(\tau).
\end{equation}

Now, we will prove that, for any $\tau$ the matrix function $M(\tau)$ is invertible. It follows from the fact that function $k_{\beta}$ has an unique root $\tau=0$. Then for $\tau\neq 0$ the determinant of $M$ is nonzero. Now, for $\tau=0$, we have that
\begin{equation}\label{tec1}
\det M(r_0(0),r_1(0),r_2(0),Z)= \frac{Z^2}{2}+ q^2(0)+p^2(0)+
2p(0).
\end{equation}

By using \eqref{estp} we have that $p(0)>0$ under the assumption $\beta<0$, then the expression \eqref{tec1} is nonzero. 

Thus, we have proved that the matrix $M$ in invertible for any $\tau$ $fixed$, then follows by \eqref{matrix} that
\begin{equation}\label{tec2}
\bold f=M^{-1}\hat {\bold F},
\end{equation}
where $M^{-1}(\tau)$ denotes the inverse of matrix function $M(\tau)$. Then, formula \eqref{tec2} defines the functions $f$, $g$ and $h$ in terms of  the initial conditions $\widetilde{u}_0$ and $\widetilde{v}_0$,  the unknown functions $\tilde u$ and $\tilde v$ and satisfy the boundary conditions \eqref{Zcondition}. Finally, \eqref{u} and \eqref{v} joint with \eqref{tec2} define the integral equation that solve the Cauchy problem \eqref{kdv3}-\eqref{Zcondition}.

\section{Energy estimates for the boundary functions $f$, $g$ and $h$}\label{5}

In this section, we will estimate the boundary vector function $\bold f=(f,g,h)$. More precisely, we will demonstrate that the previously obtained functions $f$, $g$, and $h$ belong to the trace spaces $H^{\frac{s+1}{3}}(\mathbb{R}^+)$, given the assumptions $\widetilde{u}_0, \widetilde{v}_0 \in H^s(\mathbb{R})$, and that the functions $\tilde{u}$ and $\tilde{v}$ are in $C([0, T]; H^s(\mathbb{R}))$.  To simplify the computations we will consider $Z=1$. An elementary computation yields the following formula for the matrix inverse of $M$ 
\begin{equation}\label{inverse}
\frac{1}{n_1}\left[ \begin{array}{ccc}
\frac{a_{11}}{n_2} & a_{12} & a_{13} \\ 
\frac{a_{21}}{n_2} & a_{22} & a_{23}\\
a_{31} & a_{32}  & a_{33}
\end{array} \right],
\end{equation}
where the entries of the matrix are  functions on the variable $\tau$ and they depend of the roots of algebraic equation \eqref{algebraic} and they are given by
\begin{equation}\label{n1}
\begin{cases}
 &n_1=2r_0^2-2r_0r_1-2r_0-2r_0r_2+2r_1+2r_1r_2+2r_2+1,\\
&n_2=1-2r_0^2-2r_1r_2,\\
&a_{11}=-4r_0^2r_1-4r_0^2r_2-4r_0^2r_1r_2-2r_0^2+2r_1-4r_1r_2^2-4r_1^2r_2^2-4r_1^2r_2+2r_2+1,\\
&a_{21}=-2r_0^3-2r_0^2r_1-2r_0^2r_2+2r_0r_1r_2-r_0+r_1-2r_1r_2^2-2r_1^2+2r_2-2r_1r_2+1,\\
&a_{31}=r_0(1-2r_0-2r_1r_2),\\
&a_{12}=a_{22}=2(r_1+r_2+1),\\
&a_{32}=1-2r_0^2-2r_1r_2,\\
&a_{13}=a_{23}=2,\\
&a_{33}=2(r_0-1).
\end{cases}
\end{equation}

Thus, we have the following exact expressions for the functions $f,\ g$ and $h$:
\begin{equation}\label{fgh}
\left\{\begin{array}{c}
\hat{h}(\tau)=\frac{1}{n_1}\left(\frac { a _ { 1 1 } } { n _ { 2 } } \left(\hat{F}_1(0, \tau)-\hat{F}_2(0, \tau)+a_{12}\left(-\partial_x \hat{F}_1(0, \tau)+\partial_x \hat{F}_2(0, \tau)\right)\right.\right. \\
\left.\quad+a_{13}\left(-\partial_x^2 \hat{F}_1(0, \tau)+\partial_x^2 \hat{F}_1(0, \tau)+\frac{Z^2}{2} \hat{F}_2(0, \tau)+Z \partial_x \hat{F}_2(0, \tau)\right)\right) \\
\hat{f}(\tau)=\frac{1}{n_1}\left(\frac { a _ { 2 1 } } { n _ { 2 } } \left(\hat{F}_1(0, \tau)-\hat{F}_2(0, \tau)+a_{22}\left(-\partial_x \hat{F}_1(0, \tau)+\partial_x \hat{F}_2(0, \tau)\right)\right.\right. \\
\left.\quad+a_{23}\left(-\partial_x^2 \hat{F}_1(0, \tau)+\partial_x^2 \hat{F}_1(0, \tau)+\frac{Z^2}{2} \hat{F}_2(0, \tau)+Z \partial_x \hat{F}_2(0, \tau)\right)\right) \\
\hat{g}(\tau)=\frac{1}{n_1}\left(a _ { 3 1 } \left(\hat{F}_1(0, \tau)-\hat{F}_2(0, \tau)+a_{32}\left(-\partial_x \hat{F}_1(0, \tau)+\partial_x \hat{F}_2(0, \tau)\right)\right.\right. \\
\left.\quad+a_{33}\left(-\partial_x^2 \hat{F}_1(0, \tau)+\partial_x^2 \hat{F}_1(0, \tau)+\frac{Z^2}{2} \hat{F}_2(0, \tau)+Z \partial_x \hat{F}_2(0, \tau)\right)\right).
\end{array}\right.
\end{equation}
The following lemma provides the necessary estimates for the functions that appear in \eqref{n1} at high frequencies.
\begin{lemma}\label{tec3}Let $n_i=n_i(\tau)$ given as in \eqref{n1}, ($i=1,2$) and $a_{ij}=a_{ij}(\tau)$, $(i,j\in\{1,2,3\})$ the functions given by \eqref{n1} then hold
$$
\begin{aligned}
\left|n_i(\tau)\right| & \gtrsim\langle\tau\rangle^{\frac{2}{3}},(i=1,2) \text { for all }|\tau|>2; \\
\left|a_{11}(\tau)\right| & \lesssim \langle\tau\rangle ^{4/3} \text{ for all }|\tau|>2;\\
\left|a_{21}(\tau)\right|+\left|a_{31}(\tau)\right| & \lesssim\langle\tau\rangle \text { for all }|\tau|>2; \\
\left|a_{12}(\tau)\right|+\left|a_{22}(\tau)\right|+\left|a_{33}\right| & \lesssim\langle\tau\rangle^{1 / 3} \text { for all }|\tau|>2; \\
\left|a_{32}(\tau)\right| & \lesssim\langle\tau\rangle^{1 / 3} \text { for all }|\tau|>2; \\
\left|a_{13}(\tau)\right|+\left|a_{23}(\tau)\right| & \lesssim 1 \text { for all } \tau \in \mathbb{R}.
\end{aligned}
$$
\end{lemma}
\begin{proof}
The estimates of each $a_{ij}$ follows directly from the explict formulas given by \eqref{n1} and the estimates for $r_0,$ $r_1$ and $r_2$ given in \eqref{est1}.  The proof of $n_i$ term follows from  the explicit formulas for $r_0, r_1$ and $r_2$ we have that $\operatorname{Re} n_1(\tau)=(2 p(\tau)+1)^2$. Then the estimate \eqref{estp} proves the lemma.
\end{proof}

The next step is to show that the unknown functions $f$ and $g$ are in $H^{\frac{s+1}{3}}(\mathbb{R})$ and $h$ are contained in $H^{\frac{s}{3}}(\mathbb{R})$. To do this we use the decay of functions $r_0, r_1$ and $r_2$. By using Lemma \ref{tec3} we have the following lemma.

\begin{lemma}\label{est_fgh} Assume that $\tilde u$ and $\tilde v$ are in $C\left([0, T] ; H^s(\mathbb{R})\right)$, then for the functions $f,$ $g,$ and $h$ given by the formula \eqref{fgh} hold the following estimate
\begin{equation*}
	\begin{split}
	&\|f\|_{H^{\frac{s+1}{3}}(\mathbb{R}+)}+\|h\|_{H^{\frac{s+1}{3}}(\mathbb{R}^+)}+\|g\|_{H^{\frac{s}{3}}\left(\mathbb{R}^{+}\right)} \\
	&\quad \leq\left\|F_1(0, t)\right\|_{H^{\frac{s+1}{3}}(\mathbb{R}^+)}+\left\|F_2(0, t)\right\|_{H \frac{s+1}{3}(\mathbb{R}+)}+\left\|\partial_x F_1(0, t)\right\|_{H^{\frac{s}{3}}(\mathbb{R}+)}+\left\|\partial_x F_2(0, t)\right\|_{H \frac{s}{3}(\mathbb{R}+)}\\
&\quad \quad 
+\left\|\partial_x^2 F_1(0, t)\right\|_{H^{ \frac{s-1}{3}}(\mathbb{R}+)}+\left\|\partial_x^2 F_2(0, t)\right\|_{H ^{\frac{s-1}{3}}(\mathbb{R}^+)}.
\end{split}
\end{equation*}
\end{lemma}
\begin{proof} The proof follows from the estimates
$$
\begin{array}{r}
\left|\frac{a_{11}}{n_1 n_2}\right| \lesssim 1,\left|\frac{a_{12}}{n_1}\right| \lesssim\langle\tau\rangle^{-\frac{1}{3}},\left|\frac{a_{13}}{n_1}\right| \lesssim\langle\tau\rangle^{-\frac{2}{3}} \\
\left|\frac{a_{21}}{n_1 n_2}\right| \lesssim\langle\tau\rangle^{-1},\left|\frac{a_{22}}{n_1}\right| \lesssim\langle\tau\rangle^{-\frac{1}{3}},\left|\frac{a_{13}}{n_1}\right| \lesssim\langle\tau\rangle^{-\frac{2}{3}} \\
\left|\frac{a_{31}}{n_1}\right| \lesssim\langle\tau\rangle^{1 / 3},\left|\frac{a_{32}}{n_1}\right| \lesssim 1,\left|\frac{a_{33}}{n_1}\right| \lesssim(\tau)^{-\frac{2}{3}},
\end{array}
$$
which is a direct consequence of Lemma \ref{tec3}.\end{proof}
Finally, as the consequence of Lemmas \ref{group} and \ref{duhamel}, we have the following estimate for the trace functions:
\begin{equation}\label{tec4}
\begin{split}
\left\|F_1(0, t)\right\|_{H^{ \frac{s+1}{3}}(\mathbb{R}+)}+\left\|\partial_x F_1(0, t)\right\|_{H^{\frac{s}{3}}(\mathbb{R}+)}+\left\|\partial_x^2 F_1(0, t)\right\|_{H^{\frac{s-1}{3}}(\mathbb{R}+)} \lesssim\left\|u_0\right\|_{H^s(\mathbb{R}+)}+\|u\|_{X_s^1}^2;\\
\left\|F_2(0, t)\right\|_{H^{ \frac{s+1}{3}}(\mathbb{R}+)}+\left\|\partial_x F_2(0, t)\right\|_{H^{\frac{s}{3}}(\mathbb{R}^+)}+\left\|\partial_x^2 F_2(0, t)\right\|_{H^{\frac{s-1}{3}}(\mathbb{R}+)} \lesssim\left\|v_0\right\|_{H^s(\mathbb{R}^-)}+\|v\|_{X_s^2}^2.
\end{split}
\end{equation}

\subsection{Proof of Theorem \ref{theorem1} } 
\begin{proof}
The proof will be based in the Banach's Fixed Point Theorem. By convenience in the exposition, we will consider $|\bold E_{-}|=|\bold E_{+}|=1$. Let $\bold{u}_0=(u_0, v_0) \in H^s(\mathcal G)\cap \mathcal N_{0, Z}(s)$ with $\tilde{\bold{u}}_0=(\tilde{u}_0, \tilde{v}_0) \in H^s(\mathbb R)\times H^s(\mathbb R) $ such that  $\tilde{u}_0|_{\mathbb{R}^+}= u_0$, $\tilde{v}_0|_{\mathbb R^-}=v_0$,
and $\|\tilde{u}_0\|_{ H^s(\mathbb R)}\leq c \|{u}_0\|_{ H^s(\mathbb R^+)}$, $\|\tilde{v}_0\|_{ H^s(\mathbb R)}\leq c \|{v}_0\|_{ H^s(\mathbb R^-)}$. We consider the Banach space  $X_s^1 (\mathbb R^2)$
\begin{equation}\label{tec001}
X_s^1(\mathbb R^2)=\left\{w: \mathbb R^ 2\to \mathbb R: \|w\|_{X_s^1}\equiv
N(w)<\infty\right\},
\end{equation}
where 
$$
N(w)=\text{max}\Big\{\|w\|_{C\left(\mathbb{R}_t ; H^s\left(\mathbb{R}_x\right)\right)}, \|w\|_{C\left(\mathbb{R}_x ; H^{\frac{s+1}{3}}\left(\mathbb{R}_t\right)\right)}, \|w\|_{X^{s, b, \beta, \sigma}},  \|w_x\|_{C\left(\mathbb{R}_x ; H^{\frac{s}{3}}\left(\mathbb{R}_t\right)\right)},  \|w_{xx}\|_{C\left(\mathbb{R}_x ; H^{\frac{s-1}{3}}\left(\mathbb{R}_t\right)\right)}\Big\}
$$
For $X_s=X_s^1 (\mathbb R^2) \times X_s^1(\mathbb R^2)$ and $\bold{u}=(\tilde{u}, \tilde{v})\in X_s$ we will consider $\|\bold{u}\|_{X_s}=\|\tilde{u}\|_{X_s^1}+\| \tilde{v}\|_{X_s^1}$. Next,  we define $\Lambda_{\tilde{\bold{u}}_0}: X_s\to X_s$ as
$$
\Lambda_{\tilde{\bold{u}}_0}(\bold{u})=\left(\Lambda_1(\tilde{u}), \Lambda_2(\tilde{v})\right)
$$
with
\begin{equation}\label{aux1}
\Lambda_1(\tilde{u} )(x, t)=\psi(t) S_\beta(t) \tilde{u}_0+\psi(t) \mathcal{K}\left(\psi_T^ 2(t)\tilde{u}\partial_x \tilde{u}\right)+\psi(t)\mathcal R_\beta h(x, t)
\end{equation}
and
\begin{equation}\label{aux2} \Lambda_2 \tilde{v}(x, t)=\psi(t) S_\beta(t) \tilde{v}_0+\psi(t) \mathcal{K}\left(\psi_T^ 2(t)\tilde v\partial_x\tilde{v}\right)+\psi(t)\mathcal L_\beta^1 f(x, t)+\psi(t)\mathcal L_\beta^2 g(x, t),
\end{equation}
with $R_\beta,\  L^j_\beta$ defined in \eqref{R}-\eqref{L1}-\eqref{L2}, and  $h=h(t),\ f=f(t),\ g=g(t)$ are given via the Fourier transform with regard to $t$ by the formulas in \eqref{fgh}, with $\hat{F_j}$ representing  the Fourier transform with regard to $t$ of the functions $F_j=F_j(x,t)$, $i=1,2$, given by the formulas \eqref{F1} and \eqref{F2}.

\begin{remark}
Here, we consider the extended version on all $\mathbb R^2$ of the functions $\mathcal R_{\beta} h$, $\mathcal L_{\beta}^{1}f$ and $\mathcal L_{\beta}^2g$ given in Remark \ref{ext1}. Also, we note that the operator $\Lambda$ depends of the extensions $\widetilde{u}_0$ and $\widetilde {v}_0$, but by Remark \ref{uni} we see that the restriction of the function $\Lambda \bold u\mid_ {\mathcal G \times (0,T)}=\left(\Lambda_1(\tilde{u})\mid_{\R^+\times(0,T)}, \Lambda_2(\tilde{v})\mid_{\R-\times (0,T)}\right)$ does not depend  of these extensions. Moreover, the extension of the initial data $\bold{u}_0$ is necessary to apply the keys to strategy of Bourgain \cite{Bourgain} on whole the line and of  Faminskii \cite{Faminskii,Faminskii2} on half-line, and  it do not bring problems of uniqueness of solutions.  
\end{remark}

Next, we show that there is a $\gamma >0$ such that for $\bold{u}\in \bar{B}_\gamma (\bold{0})=\{\bold{w}\in X_s: \|\bold{w}\|_{X_s}\leq \gamma\}$, we have $\Lambda_{\tilde{\bold{u}}_0}(\bold{u}) \in B_\gamma (\bold{0})$ and $\Lambda_{\tilde{\bold{u}}_0}: \bar{B}_\gamma (\bold{0})\to \bar{B}_\gamma (\bold{0})$ is a contraction.

 Firstly, we start by estimating  $\|\Lambda_1(\tilde{u})\|_{X_s^1}$. We have from definition of $F_i$ at \eqref{F1} and \eqref{F2}, Lemmas \ref{tec3}-\ref{est_fgh}, and Lemmas \ref{group}, \ref{duhamel}, \ref{psiT} and \ref{bilinear}  that for $h$ in \eqref{fgh}
\begin{equation}\label{0estima1}
 \begin{array}{ll}
&\|h\|_{H^{\frac{s+1}{3}}(\mathbb R_t)}\leq  D_1 \sum_{j=0}^2 \|\partial^j_x F_i(0, \tau)\|_{H^{\frac{s+1-j}{3}}(\mathbb R_t)}\leq D_1 \sum_{j=0}^2 \|\partial^j_x F_i(x, \tau)\|_{C(\mathbb R_x: H^{\frac{s+1-j}{3}}(\mathbb R_t))}\\
\\
&\quad\leq  C_0D_1(\|\tilde{u}_0\|_{H^s(\mathbb R)}+ \|\tilde{v}_0\|_{H^s(\mathbb R)})  + D_2 (T^r)^2(\|\tilde{u}\|^2_{X^{s, b,\beta, \sigma}}+\|\tilde{v}\|^2_{X^{s, b,\beta, \sigma}})\\
\\
& \quad\leq C_3\|\bold{{u}}_0\|_{H^s(\mathcal G)}+ D_3 (T^r)^2\|\bold{u}\|_{X_s}^2,
\end{array}
\end{equation}
with $D_i$, $C_0$ and $C_3$ are generic positive constants depending of $Z$ and $s$ which appear in Lemma \eqref{fgh} and \eqref{tec4}, and the power $r>0$ of $T$ is obtained in Lemma \ref{psiT} or Remark \ref{power}. Thus, from Lemmas \eqref{group}, \eqref{duhamel}, \eqref{psiT}, Lemma 3.3 and \eqref{0estima1} we get,
\begin{equation}\label{1estima2}
 \begin{array}{ll}
&\|\psi(t) S_\beta(t) \tilde{u}_0 + \psi(t) \mathcal K(\psi^2_T\partial_x(\tilde{u}\tilde{u})) +\psi(t) R_\beta h\|_{C(\mathbb R_t;H^s (\mathbb R_x))}\\
\\
&\quad\quad\leq C_0\|\tilde{u}_0\|_{H^s(\mathbb R)} +
C_1 \| \psi^2_T\partial_x(\tilde{u}\tilde{u}))\|_{Y^{s, b, \beta, \sigma}} + C_2\|h\|_{H^{\frac{s+1}{3}}(\mathbb R_t)}\\
&\quad\quad\leq  C_3\|\bold{{u}}_0\|_{H^s(\mathcal G)}+D_4 (T^r)^2\|\bold{u}\|^2_{X_s}.
\end{array}
\end{equation}
Similarly, we have
\begin{equation}\label{2estima3}
 \begin{array}{ll}
&\sum_{j=0}^2\| \psi(t) \partial_x^j S_\beta(t) \tilde{u}_0 +\psi(t)\partial_x^j \mathcal K(\psi^2_T\partial_x(\tilde{u}\tilde{u})) +\psi(t) \partial_x^j R_\beta h\|_{C(\mathbb R_x;H^{\frac{s+1-j}{3}} (\mathbb R_t))}\\
\\
&\quad \leq C_0\|\tilde{u}_0\|_{H^s(\mathbb R)} +C_1 \| \psi^2_T\partial_x(\tilde{u}\tilde{u}))\|_{Y^{s, b, \beta, \sigma}} + C_2\|h\|_{H^{\frac{s+1}{3}}(\mathbb R_t)}\\
&\quad\leq  C_3\|\bold{{u}}_0\|_{H^s(\mathcal G)}+D_4 (T^r)^2\|\bold{u}\|^2_{X_s}
\end{array}
\end{equation}
and  
\begin{equation}\label{3estima4}
 \begin{array}{ll}
&\| \psi(t) S_\beta(t) \tilde{u}_0 +\psi(t)\mathcal K(\psi^2_T\partial_x(\tilde{u}\tilde{u})) +\psi(t)R_\beta h\|_{X^{s, b,\beta, \sigma}}\\
\\
& \quad \leq C_0\|\tilde{u}_0\|_{H^s(\mathbb R)} +C_1 \| \psi^2_T\partial_x(\tilde{u}\tilde{u}))\|_{Y^{s, b, \beta, \sigma}} + C_2\|h\|_{H^{\frac{s+1}{3}}(\mathbb R_t)}\\
&\quad\leq    C_3\|\bold{{u}}_0\|_{H^s(\mathcal G)}+D_4 (T^r)^2\|\bold{u}\|^2_{X_s}.
\end{array}
\end{equation}
Therefore, from \eqref{1estima2}-\eqref{2estima3}-\eqref{3estima4} follows
$$
\|\Lambda_1(\tilde{u})\|_{X_s^1}\leq   C_3\|\bold{{u}}_0\|_{H^s(\mathcal G)} +D_4 (T^r)^2\|\bold{u}\|^2_{X_s}.
$$
Similarly, we obtain $\|\Lambda_2(\tilde{v})\|_{X_s^1}\leq  C_3\|\bold{{u}}_0\|_{H^s(\mathcal G)} +D_4 (T^r)^2\|\bold{u}\|^2_{X_s}$. Then,
$$
\|\Lambda_{\tilde{\bold{u}}_0} (\bold u)\|_{X_s}\leq C_3\|\bold u_0\|_{H^s(\mathcal G)} +2D_4(T^r)^2\|\bold u\|_{X_s}^2. 
$$
Thus, by choosing first $\gamma= 2 C_3\|\bold u_0\|_{H^s(\mathcal G)}$ and then $T$ such that $2D_4(T^r)^2\gamma<\frac{1}{16}$, we obtain $\|\Lambda_{\tilde{\bold{u}}_0}(\bold u)\|_{X_s}\leq \gamma$. Therefore, $\Lambda_{\tilde{\bold{u}}_0} ( \bar{B}_\gamma (\bold{0}))\subset  \bar{B}_\gamma (\bold{0})$.

Now, to show that $\Lambda_{\tilde{\bold{u}}_0} $ is a contraction on $\bar{B}_\gamma (\bold{0})$, we argue as above and we  get for $ \bold u, \bold w\in  \bar{B}_\gamma (\bold{0})$
\begin{equation}\label{contraction}
\|\Lambda _{\tilde{\bold{u}}_0} (\bold u)-\Lambda_{\tilde{\bold{u}}_0} (\bold w)\|_{X_s}\leq 8\gamma D_4(T^r)^2 \|\bold u-\bold w\|_{X_s}<\frac14 \|\bold u-\bold w\|_{X_s}.
\end{equation}
By convenience of the reader, we show the former estimate in the case of  $\|\Lambda_1 (\tilde u)-\Lambda_1 ( w_1)\|_{C(\mathbb R_t;H^s (\mathbb R_x))}$ with $\bold w=(w_1,w_2)$. Indeed, initially we have
\begin{equation}\label{4estima5}
 \begin{array}{ll}
&\|\psi(t) \mathcal K(\psi^2_T\partial_x(\tilde{u}^2-{w_1}^2))\|_{C(\mathbb R_t;H^s (\mathbb R_x))} \leq
C_1 \| \psi^2_T\partial_x[(\tilde{u}-w_1)(\tilde{u}+w_1)]\|_{Y^{s, b, \beta, \sigma}} \\
\\
&\quad\leq  D_4 (T^r)^2\|\tilde{u}-w_1\|_{X^s} \|\tilde{u}+w_1\|_{X^s}\leq 2\gamma D_4 (T^r)^2\|\bold{u}-\bold w\|_{X_s},
\end{array}
\end{equation}
and 
\begin{equation*}
\|\psi(t) R_\beta (h-j)\|_{C(\mathbb R_t;H^s (\mathbb R_x))} \leq
C_2 \| h-j\|_{H^{\frac{s+1}{3}}(\mathbb R_t)} 
\end{equation*}
with $j$ defined similarly as $h$ in \eqref{fgh} with $F_1$ and $F_2$ changed, respectively, by $$
F_3(\tilde u_0, w_1,x,t)=\psi(t)S_{\beta}(t)\tilde u_0+\psi(t)\mathcal{K}(\psi^2_T\partial_x({w_1})^2)$$  and 
$$F_4(\tilde u_0, w_2,x,t)=\psi(t)S_{\beta}(t)\tilde v_0+\psi(t)\mathcal{K}(\psi^2_T\partial_x({w_2})^2).$$ Thus, by following a similar argument
as in \eqref{0estima1} and \eqref{4estima5} (with the same generic constants) we obtain for $\mathcal Z= C(\mathbb R_x;H^{\frac{s+1-j}{3}} (\mathbb R_t))$
\begin{equation}\label{5estima6}
 \begin{array}{ll}
&\|h-j\|_{H^{\frac{s+1}{3}}(\mathbb R_t)}\leq  D_1 \sum_{j=0}^2 \|\psi(t)\partial^j_x \mathcal K(\psi^2_T\partial_x(\tilde{u}^2-{w_1}^2))\|_{\mathcal Z} +  D_1 \sum_{j=0}^2 \|\psi(t)\partial^j_x \mathcal K(\psi^2_T\partial_x(\tilde{v}^2-{w_2}^2))\|_{\mathcal Z}
\\
\\
&\leq  D_3 (T^r)^2\|\tilde{u}-w_1\|_{X^s} \|\tilde{u}+w_1\|_{X^s} + D_3 (T^r)^2\|\tilde{v}-w_2\|_{X^s} \|\tilde{v}+w_2\|_{X^s}  \leq 2\gamma D_3 (T^r)^2\|\bold{u}-\bold w\|_{X_s}.
\end{array}
\end{equation} 
Therefore,
$$
\|\Lambda_1 (\tilde u)-\Lambda_1 ( w_1)\|_{C(\mathbb R_t;H^s (\mathbb R_x))}\leq 4\gamma D_4 (T^r)^2\|\bold{u}-\bold w\|_{X_s}.
$$
Thus, by \eqref{contraction} there is a unique $\bold{u}\in \bar{B}_\gamma (\bold{0})\subset X_s$ such that $\Lambda_{\tilde{\bold{u}}_0}(\bold{u})=\bold{u}$. Hence,  as the linear  restriction-mapping 
$$
 \begin{array}{ll}
\mathrm R: X_s&\to X_s(T)\\
\hskip0.3in\bold{w}& \mapsto \bold{w}|_{\mathcal G \times [0,T]}
\end{array}
$$
satisfies $\|\mathrm R \bold{w}\|_{X_s(T)}\leq 2 \|\bold{w}\|_{X_s}$, we obtain that 
\begin{equation}\label{solution}
(u_{\bold{e}})_{\bold{e}\in \bold E}=\bold{u}|_{\mathcal G\times (0,T)}\in X_s(T)
\end{equation}
 with $T<1$,  represents  a strong solution of \eqref{kdv3} and satisfies \eqref{Zcondition}, moreover,  $
 (u_{\bold{e}}(\cdot, 0))_{\bold{e}\in \bold E}=\bold{u}_0$.   Similarly, using the argument leading to \eqref{contraction},  we obtain for $T_0\in (0, T)$ that
\begin{equation}\label{LIP}
\|\Lambda _{\tilde{\bold{u}}_0} (\bold u)-\Lambda_{\tilde{\bold{p}}_0} (\bold p )\|_{X_s}\leq C_3\| \tilde{\bold{u}}_0- \tilde{\bold{p}}_0\|_{H^ s(\mathbb R)\times H^ s(\mathbb R)}+ 8\gamma D_4(T_0^r)^ 2\|\bold u-\bold p\|_{X_s}.
\end{equation}
Thus, for $U_0=B_\delta(\tilde{\bold{u}}_0)\subset H^ s(\mathbb R)\times H^ s(\mathbb R)$ with $\delta>0$ such  that $C_3\delta+8\gamma D_4 (T_0^r)^2<\frac{1}{4}$, we get  for  $\tilde{\bold{p}}_0 \in U_0$ that $\bold p$ is a fixed point for $\Lambda_{\tilde{\bold{p}}_0} $,  and so  the map $\tilde{\bold{p}}_0\in U_0\mapsto \bold p\in X_s$ is Lipschitz. We also note that by using standard arguments, we can obtain uniqueness of $\bold{u}$ in the class $X_s$ (see Kenig, Ponce and Vega \cite{KPV}).

Next, let $T_0\in (0, T)$ and we consider for $\delta_0>0$ (to be chosen)  the open ball $W_0=B_{\delta_0}(\bold{u}_0)\cap \mathcal N_{0, Z}(s)\subset H^s(\mathcal G)$. For $\bold{p}_0\in W_0$ we consider a extension $\tilde{\bold{p}}_0\in  H^ s(\mathbb R)\times H^ s(\mathbb R)$ of $\bold{p}_0$  such that $\| \tilde{\bold{u}}_0- \tilde{\bold{p}}_0\|_{H^ s(\mathbb R)\times H^ s(\mathbb R)}\leq 2 \| \bold{u}_0- \bold{p}_0\|_{H^s(\mathcal G)}$. Therefore, by choosing $\delta_0$ such that $2\delta_0<\delta$ we obtain $\tilde{\bold{p}}_0\in  U_0$ and for $\bold{q}\in X_s$ such that $\mathrm R \bold{p}\in X_s(T_0)$ is the solution of \eqref{kdv3} satisfying \eqref{Zcondition} and  $\mathrm R \bold{p}(\cdot, 0)=\bold{p}_0$, we have from \eqref{LIP}
$$
\|\mathrm R \bold{p}- \mathrm R \bold{u} \|_{X_s(T_0)}\leq 2 \|\bold{p}-\bold{u}\|_{X_s}\leq 2C_3 \| \tilde{\bold{u}}_0- \tilde{\bold{p}}_0\|_{H^ s(\mathbb R)\times H^ s(\mathbb R)}\leq 4C_3 \| \bold{u}_0- \bold{p}_0\|_{H^s(\mathcal G)}.
$$
This shows that the mapping data-solution $\bold{p}_0\in W_0 \mapsto \mathcal R \bold{p}\in X_s(T_0)$ is Lipschitz. This finishes the proof.
\end{proof} 
\begin{remark}
	As the consequence of our proof we have an exact formula for the  group associated to the Airy operator $ {A_Z}$. More precisely, we can define $\bold S(t):D(A_Z)\subset H^3(\mathcal G)\rightarrow L^2(\mathcal G)$ by 
	\begin{equation}\label{grupo}
		\bold S(t)\bold w_0=\left ((S_{\beta}(t)\tilde u_0+R_{\beta}h(x,t))\bigg|_{\mathbb R ^-\times (0,T)},( S_{\beta}(t)\tilde v_0+L_{\beta}^1f(x,t)+L_{\beta}^2g(x,t))\bigg|_{\mathbb R ^+\times (0,T)}\right),
	\end{equation}
	where $\bold w_0=(u_0,v_0)$ and 
	\begin{equation}\label{fgh2}
		\left\{\begin{array}{c}
			\hat{h}(\tau)=\frac{1}{n_1}\left(\frac { a _ { 1 1 } } { n _ { 2 } } \left(\hat{F}_1(0, \tau)-\hat{F}_2(0, \tau)+a_{12}\left(-\partial_x \hat{F}_1(0, \tau)+\partial_x \hat{F}_2(0, \tau)\right)\right.\right. \\
			\left.\quad+a_{13}\left(-\partial_x^2 \hat{F}_1(0, \tau)+\partial_x^2 \hat{F}_1(0, \tau)+\frac{Z^2}{2} \hat{F}_2(0, \tau)+Z \partial_x \hat{F}_2(0, \tau)\right)\right) \\
			\hat{f}(\tau)=\frac{1}{n_1}\left(\frac { a _ { 2 1 } } { n _ { 2 } } \left(\hat{F}_1(0, \tau)-\hat{F}_2(0, \tau)+a_{22}\left(-\partial_x \hat{F}_1(0, \tau)+\partial_x \hat{F}_2(0, \tau)\right)\right.\right. \\
			\left.\quad+a_{23}\left(-\partial_x^2 \hat{F}_1(0, \tau)+\partial_x^2 \hat{F}_1(0, \tau)+\frac{Z^2}{2} \hat{F}_2(0, \tau)+Z \partial_x \hat{F}_2(0, \tau)\right)\right) \\
			\hat{g}(\tau)=\frac{1}{n_1}\left(a _ { 3 1 } \left(\hat{F}_1(0, \tau)-\hat{F}_2(0, \tau)+a_{32}\left(-\partial_x \hat{F}_1(0, \tau)+\partial_x \hat{F}_2(0, \tau)\right)\right.\right. \\
			\left.\quad+a_{33}\left(-\partial_x^2 \hat{F}_1(0, \tau)+\partial_x^2 \hat{F}_1(0, \tau)+\frac{Z^2}{2} \hat{F}_2(0, \tau)+Z \partial_x \hat{F}_2(0, \tau)\right)\right),
		\end{array}\right.
	\medskip
	\end{equation}
where, the functions $F_1$ and $F_2$ are given by
	$$F_1(\tilde u_0,x,t)=S_{\beta}(t)\tilde u_0$$ and $$F_2(\tilde v_0, x,t)=S_{\beta}(t)\tilde v_0.$$
	\end{remark}
\begin{remark}
{Since the Airy operator $A_Z$ is skew self adjoint  on $D( A_Z)$ follows from the theory of semigroups that $A_Z$ is  the generator of a unique unitary group $\{\bold S(t)\}$ on $L^2(\mathcal G)$. Then, we have that  $\bold S(t)$ defined in \eqref{grupo} is the only group associated to Airy operator $A_Z$.}

Also, combining \eqref{aux1} and \eqref{aux2}  with \eqref{grupo} we point out that it is possible to rewrite the Duhamel  integral formula associated to \eqref{kdv3}, for $t\in [0,T]$ as the following
\begin{equation}\label{Duhamel}
	\bold u(t)=\bold S(t)\bold u_0+\int_0^t\bold S(t-t')(\bold u\bold u_x)dt' ,
	\end{equation}
where we denote $\bold u:=(u,v)$.
\end{remark}

\subsection{Proof of Theorem \ref{dependence} } 

\begin{proof} The proof will be based in the Implicit Function Theorem and our estimates used in the proof of Theorem \ref{theorem1}. By convenience in the exposition, we will consider $|\bold E_{-}|=|\bold E_{+}|=1$. Let $\bold{u}_0=(u_0, v_0) \in H^1(\mathcal G)\cap \mathcal N_{0, Z}(1)$, and $\bold{u}\in X_1(T)$, $T=T(\|\bold{u}_0 \|_1)$, the unique solution of the IVP \eqref{kdv3}-\eqref{Zcondition}, on the class $X_s(T)$, given by Theorem \ref{theorem1} with initial data $\bold{u}_0$. Moreover, for $\bold{u}=(u, v)=\left(\left.\tilde{u}\right|_{\mathbb{R}^{+}},\left.\tilde{v}\right|_{\mathbb{R}^{-}}\right)$, where $\tilde{\bold{u}}=(\tilde{u}, \tilde{v})$ is the fixed point obtained for $\Lambda=(\Lambda_1, \Lambda_2)$ on $X_1$ and $\Lambda_i$ determined by \eqref{aux1}-\eqref{aux2}.

Next,  let $\mathcal W\subset H^1(\mathbb R)\times H^1(\mathbb R) $ be a neighborhood of the extension $\tilde{\bold{u}}_0=(\tilde{u}_0, \tilde{v}_0)$ of $\bold{u}_0$ such that  $\tilde{u}_0|_{\mathbb R^+}= u_0$, $\tilde{v}_0|_{\mathbb R^+}=v_0$,
and $\|\tilde{u}_0\|_{ H^1(\mathbb R)}\leq c \|{u}_0\|_{ H^1(\mathbb R^+)}$, $\|\tilde{v}_0\|_{ H^1(\mathbb R)}\leq c \|{v}_0\|_{ H^1(\mathbb R^-)}$.  For $X_1=X_1^1(\mathbb R^2) \times X_1^1(\mathbb R^2)$ with $ X_1^1(\mathbb R^2)$ determined in \eqref{tec001},  we define the mapping  
$$
\mathcal H: \mathcal W\times X_1\to X_1
$$
for $(\bold{w}_0, \bold{w})\in \mathcal W\times X_1$ as 
$$
\mathcal H(\bold{w}_0, \bold{w})(t)\equiv \bold{w}(t)-\Big(\psi(t)\bold S_\beta(t){\bold w}_0+ \psi(t)D(\bold{w})\Big),
$$
where for $\bold{w}_0=(\phi_0, \psi_0)$ we have
\begin{equation}\label{grupo2}
\bold S_\beta(t)\bold w_0=\Big (S_\beta(t){\phi}_0, S_\beta(t){\psi}_0\Big),
\end{equation}
and $D(\bold{w})$ is defined for $\bold{w}=(w_1, w_2) \in X_1$ by 
\begin{equation}\label{duha2}
D({\bold{w}})=\Big (\mathcal K(\psi^2_T(t){w}_1\partial_x{w}_1)+ R_\beta H , \mathcal K(\psi^2_T(t){w}_2\partial_x {w}_2)+  L^1_\beta F +  L^2_\beta G\Big)
\end{equation}
with $ R_\beta,  L^j_\beta$ defined in \eqref{R}-\eqref{L1}-\eqref{L2}, with $H=H(t), F=F(t), G=G(t)$ given via the Fourier transform with regard to $t$ by the formulas in \eqref{fgh}, with $\hat{F_j}$ representing  the Fourier transform with regard to $t$ of the function $F_j=F_j(x,t)$, $i=1,2$, determined for $(x,t)\in \mathbb R$ by
\begin{equation}\label{Fj}
\left\{\begin{array}{cc}
F_1(x,t)\equiv F_1({\phi}_0, {w}_1, x,t)&=\psi(t)S_\beta(t){\phi}_0 (x)+ \psi(t)\mathcal K(\psi^2_T{w}_1\partial_x{w}_1)(x,t),\\
F_2(x,t)\equiv F_2({\psi}_0, {w}_2, x,t)&=\psi(t)S_\beta(t){\psi}_0(x) +\psi(t)\mathcal K(\psi^2_T{w}_2\partial_x{w}_2)(x,t).
\end{array}\right.
\end{equation}
Thus, by the analysis in the proof of Theorem \ref{theorem1} we get  for the fixed point $\tilde {\bold{u}}=(\tilde {{u}}, \tilde {{v}})$ of $\Lambda$ that
$\mathcal H(\tilde{\bold{u}}_0, \tilde {\bold{u}})=\bold {0}$.

Next, we show that the linear application $D_2 \mathcal H(\tilde{\bold{u}}_0, \tilde{\bold{u}})\equiv \partial_\bold{w} \mathcal H(\tilde{\bold{u}_0}, \tilde{\bold{u}})$ is invertible. Initially, we determined one formula for 
$$
D_2 \mathcal H(\tilde{\bold{u}}_0, \tilde{\bold{u}})(\bold{w})=\frac{d}{d\epsilon} \mathcal H(\tilde{\bold{u}}_0, \tilde{\bold{u}}+\epsilon \bold{w})|_{\epsilon=0}
$$
with $\bold{w}\in X_1$. Thus, we get
\begin{equation}\label{1deriva}
D_2 \mathcal H(\tilde{\bold{u}}_0, \tilde{\bold{u}})(\bold{w})= \bold{w}- \psi (t)\frac{d}{d\epsilon} D(\tilde{\bold{u}}+\epsilon {\bold{w}})|_{\epsilon=0}.
\end{equation}
Next, by \eqref{duha2} we get the following first relation for $\bold{w}=(w_1, w_2)$
\begin{equation}\label{deriK}
\frac{d}{d\epsilon} 
 (\frac12 \mathcal K( \psi^2_T\partial_x(\tilde{u}+\epsilon {w}_1)^2), \frac12 \mathcal K(\psi^2_T \partial_x(\tilde{v}+\epsilon {w}_2)^2)|_{\epsilon=0}= (\mathcal K(\psi^2_T\partial_x(\tilde{u}{w}_1)), \mathcal K(\psi^2_T\partial_x(\tilde{v}{w}_2)).
\end{equation}
Now, by \eqref{Fj} we get for $F_{1, \epsilon} =F_1({\phi}_0, \tilde{u}+\epsilon {w}_1, x,t)$ and  $F_{2, \epsilon}=F_2({\psi}_0, \tilde{v}+\epsilon {w}_2, x,t)$ that
\begin{equation}\label{deriFj}
\left\{\begin{array}{cc}
\frac{d}{d\epsilon} \hat{F}_{1, \epsilon} (x, \tau)|_{\epsilon=0}&=[\psi (t)\mathcal K(\psi^2_T \partial_x(\tilde{u}{w}_1))]^{\wedge}(x, \tau)\equiv \widehat{M}_1(x, \tau)\\
\frac{d}{d\epsilon} \hat{F}_{2, \epsilon} (x, \tau)|_{\epsilon=0}&=[\psi (t)\mathcal K(\psi^2_T \partial_x(\tilde{v}{w}_2))]^{\wedge}(x, \tau)\equiv \widehat{M}_2(x, \tau)
\end{array}\right.
\end{equation}
Thus, we establish that
\begin{equation}\label{tracej}
\left\{\begin{array}{cc}
\widehat{M}_1(0, \tau)&=[\psi (t)\mathcal K(\psi^2_T\partial_x(\tilde{u}{w}_1))]^{\wedge}(0, \tau)\\
\widehat{M}_2(0, \tau)&= [\psi (t)\mathcal K(\psi^2_T\partial_x(\tilde{v}{w}_2))]^{\wedge}(0, \tau).
\end{array}\right.
\end{equation}
Similarly, we obtain from the notation $\partial_x^jJ(0, \tau)=\lim_{x\to 0} \partial_x^jJ (x, \tau)$, that for $j=1,2$, 
\begin{equation}\label{tracej}
\left\{\begin{array}{cc}
\widehat{\partial_x^j M}_1(0, \tau)&= \frac{d}{d\epsilon} \partial_x^j\hat{F}_{1, \epsilon} (0, \tau)|_{\epsilon=0}= [\psi (t)\partial_x^j \mathcal K(\psi^2_T\partial_x(\tilde{u}{w}_1))]^{\wedge}(0, \tau)\\
\widehat{\partial_x^j M}_2(0, \tau)&= \frac{d}{d\epsilon} \partial_x^j\hat{F}_{2, \epsilon} (0, \tau)|_{\epsilon=0}= [\psi (t)\partial_x^j \mathcal K(\psi^2_T\partial_x(\tilde{v}{w}_2))]^{\wedge}(0, \tau),
\end{array}\right.
\end{equation}
we recall that ``\;$\widehat{}$\; '' represents the Fourier transform with regard to the time. Therefore,  for $M=M(t)$ we obtain 
$$
\frac{d}{d\epsilon} R_\beta H |_{\epsilon=0}(x,t)= R_\beta M (x,t),
$$
 where 
 \begin{equation}\label{M}
 \begin{array}{cc}
 \widehat{M}(\tau)&=\frac{1}{n_1}\big(\frac { a _ { 1 1 } } { n _ { 2 } } \big( \widehat{M}_1(0, \tau)- \widehat{M}_2(0, \tau)+a_{12}\big(-\widehat{\partial_x M}_1(0, \tau)+\widehat{\partial_x M}_2(0, \tau)\big)\\
&+a_{13}\big(-\widehat{\partial_x^2 M}_1(0, \tau)+\widehat{\partial_x^2 M}_2(0, \tau)+\frac{Z^2}{2} \widehat{M}_2(0, \tau)+Z \widehat{\partial_x M}_2(0, \tau)\big)\big).
\end{array}
\end{equation}
Similarly, for $N=N(t)$ and  $P=P(t)$ we have
$$
\frac{d}{d\epsilon} \mathcal L^1_\beta F |_{\epsilon=0}(x,t)= \mathcal L^1_\beta N (x,t), \quad \frac{d}{d\epsilon} \mathcal L^2_\beta G |_{\epsilon=0}(x,t)= \mathcal L^2_\beta P (x,t)
$$
 where 
 \begin{equation}\label{N}
 \begin{array}{cc}
 \widehat{N}(\tau)&=\frac{1}{n_1}\big(\frac { a _ { 2 1 } } { n _ { 2 } } \big(\widehat{M}_1(0, \tau)-\widehat{M}_2(0, \tau)+a_{22}\big(-\widehat{\partial_x M}_1(0, \tau)+\widehat{\partial_x M}_2(0, \tau)\big)\\
&+a_{23}\big(-\widehat{\partial_x^2 M}_1(0, \tau)+\widehat{\partial_x^2 M}_2(0, \tau)+\frac{Z^2}{2} \widehat{M}_2(0, \tau)+Z \widehat{\partial_x M}_2(0, \tau)\big)\big),\\
\widehat{P}(\tau)&=\frac{1}{n_1}\big( a _ { 3 1 }  \big(\widehat{M}_1(0, \tau)-\widehat{M}_2(0, \tau)+a_{32}\big(-\widehat{\partial_x M}_1(0, \tau)+\widehat{\partial_x M}_2(0, \tau)\big)\\
&+a_{33}\big(-\widehat{\partial_x^2 M}_1(0, \tau)+\widehat{\partial_x^2 M}_2(0, \tau)+\frac{Z^2}{2} \widehat{M}_2(0, \tau)+Z \widehat{\partial_x M}_2(0, \tau)\big)\big).
\end{array}
\end{equation}
Therefore, from \eqref{1deriva}-\eqref{deriK}-\eqref{M}-\eqref{N} we get
\begin{equation}\label{todaderi}
D_2 \mathcal H(\tilde{\bold{u}}_0, \tilde{\bold{u}})(\bold{w})= \bold{w}- \psi(t)(\mathcal K(\psi^2_T\partial_x(\tilde{u}{w}_1)) +R_\beta M, \mathcal K(\psi^2_T\partial_x(\tilde{v}{w}_2))+ \mathcal L^1_\beta N + \mathcal L^2_\beta P).
\end{equation}

Now, we show that $\| D_2 \mathcal H(\tilde{\bold{u}}_0, \tilde{\bold{u}}) -I\|_{X_1}<1$ and therefore $D_2 \mathcal H(\tilde{\bold{u}}_0, \tilde{\bold{u}})$ will be invertible. Thus, we need to  estimate every component in \eqref{todaderi} in the norm of $X_1^1(\mathbb R^ 2)$. For  the first component, by \eqref{M}, definition of $M_i$ at \eqref{deriFj},  Lemmas \eqref{tec3} \eqref{est_fgh}, and Lemmas \eqref{duhamel}, \eqref{psiT} and \eqref{bilinear}  we obtain for $s=1$
\begin{equation}\label{estima1}
 \begin{array}{ll}
\|M\|_{H^{\frac{s+1}{3}}(\mathbb R_t)}&\leq  D_1 \sum_{j=0}^2 \|\partial^j_x M_i(0, \tau)\|_{H^{\frac{s+1-j}{3}}(\mathbb R_t)}\leq D_1 \sum_{j=0}^2 \|\partial^j_x M_i(x, \tau)\|_{C(\mathbb R_x: H^{\frac{s+1-j}{3}}(\mathbb R_t))}\\
&\leq  D_2 (T^r)^2\|\tilde{\bold{u}}\|_{X^{s, b,\beta, \sigma}}\|\bold{w}\|_{X^{s, b,\beta, \sigma}}\leq D_2 (T^r)^2\|\tilde{\bold{u}}\|_{X_1}\|\bold{w}\|_{X_1} .
\end{array}
\end{equation}
with $D_i$ generic positive constants depending of $Z$ and from the estimative-constants in Lemmas \eqref{duhamel}, \eqref{psiT} and \eqref{bilinear}. Thus, from  Lemmas \eqref{duhamel}, \eqref{psiT}, \eqref{bilinear} and Lemma \eqref{slr} and \eqref{estima1} we get for $s=1$,
\begin{equation}\label{estima2}
 \begin{array}{ll}
\| \psi(t) \mathcal K(\psi^2_T\partial_x(\tilde{u}{w}_1)) +\psi(t) R_\beta M\|_{C(\mathbb R_t;H^s (\mathbb R_x))}&\leq C_1 \| \psi^2_T\partial_x(\tilde{u}{w}_1))\|_{Y^{s, b, \beta, \sigma}} + C_2\|M\|_{H^{\frac{s+1}{3}}(\mathbb R_t)}\\
&\leq D_3 (T^r)^2\|\tilde{\bold{u}}\|_{X_1}\|\bold{w}\|_{X_1}.
\end{array}
\end{equation}
Similarly we obtain for $s=1$ that
\begin{equation}\label{estima3}
 \begin{array}{ll}
\sum_{j=0}^2\| \psi(t)\partial_x^j \mathcal K(\psi^2_T\partial_x(\tilde{u}{w}_1)) +\psi(t) \partial_x^j R_\beta M\|_{C(\mathbb R_x;H^{\frac{s+1-j}{3}} (\mathbb R_t))}&\leq C_1 \| \psi^2_T\partial_x(\tilde{u}{w}_1))\|_{Y^{s, b, \beta, \sigma}} + C_2\|M\|_{H^{\frac{s+1}{3}}(\mathbb R_t)}\\
&\leq D_3 (T^r)^2\|\tilde{\bold{u}}\|_{X_1}\|\bold{w}\|_{X_1}.
\end{array}
\end{equation}
 Next, we obtain again by Lemmas \eqref{duhamel},\eqref{psiT} \eqref{slr} and \eqref{estima1} that for $s=1$
\begin{equation}\label{estima4}
 \begin{array}{ll}
\| \psi(t)\mathcal K(\psi^2_T\partial_x(\tilde{u}{w}_1)) +\psi(t)R_\beta M\|_{X^{s, b,\beta, \sigma}}&\leq C_1 \| \psi^2_T\partial_x(\tilde{u}{w}_1))\|_{Y^{s, b, \beta, \sigma}} + C_2\|M\|_{H^{\frac{s+1}{3}}(\mathbb R_t)}\\
&\leq D_3 (T^r)^2\|\tilde{\bold{u}}\|_{X_1}\|\bold{w}\|_{X_1}.
\end{array}
\end{equation}
Therefore from \eqref{estima2}-\eqref{estima3}-\eqref{estima4} we have for the first component,
$$
\| \psi(t)\mathcal K(\psi^2_T\partial_x(\tilde{u}{w}_1)) +\psi(t)R_\beta M\|_{X_1^1}\leq D_3 (T^r)^2\|\tilde{\bold{u}}\|_{X_1}\|\bold{w}\|_{X_1}.
$$
Similarly for the  second component,
$$
\| \mathcal K(\psi^2_T\partial_x(\tilde{v}{w}_2))+ \mathcal L^1_\beta N + \mathcal L^2_\beta P\|_{X_1^1}\leq D_3 (T^r)^2\|\tilde{\bold{u}}\|_{X_1}\|\bold{w}\|_{X_1}.
$$
Now, by the proof of Theorem 1.1, we know $\tilde{\bold{u}}\in B_\gamma( \bold{0})=\{\bold{p}\in X_1: \|\bold{p}\|_{X_1}<\gamma\}$ ($\gamma=2c\|\bold{u}_0\|_{H^1(\mathcal G)}$) and $ 2D_3 (T^r)^2\gamma<\frac{1}{16}$. Therefore,
$$
\| D_2 \mathcal H(\tilde{\bold{u}}_0, \tilde{\bold{u}})(\bold{w})-\bold{w}\|_{X_1}\leq 2D_3 (T^r)^2\|\tilde{\bold{u}}\|_{X_1}\|\bold{w}\|_{X_1}<\frac{1}{16}\|\bold{w}\|_{X_1}.
$$
Therefore, from operator's theory we conclude that $D_2 \mathcal H(\tilde{\bold{u}}_0, \tilde{\bold{u}})$ is invertible. 

Next, it is not difficult to show that $\mathcal H$ is of class $C^2$. Therefore, there exists a unique continuous map of class $C^2$, $\Phi:\mathcal W_0\to X_1$, defined on an open ball $\mathcal W_0=B_\delta(\tilde{\bold{u}}_0)$ of $\tilde{\bold{u}}_0$ such that $\Phi(\tilde{\bold{u}}_0)= \tilde{\bold{u}}$ and $\mathcal H(\bold{z}_0, \Phi(\bold{z}_0))=\bold{0}$ for all $\bold{z}_0=(p_0, q_0)\in \mathcal W_0$. Hence, for $\bold{z}= \Phi(\bold{z}_0)$ it follows from the definition of the functional $\mathcal H$ that $\bold{z}$ satisfies the equation
$$
\bold{z}(t)=\psi(t)\bold S_\beta(t){\bold z}_0+ \psi(t)D(\bold{z}),\quad t\in \mathbb R
$$
with $D(\bold{z})$ defined for $\bold{z}=(z_1,z_2)\in X_1$ by \eqref{duha2}, and $H, F,G$ by the formulas in \eqref{fgh}, with $\hat{F_j}$ determined by $F_1(p_0, z_1, x, t)$ and $F_2(q_0, z_2, x, t)$ in \eqref{Fj}.

In the following we show that the following mapping data-solution associated to \eqref{kdv3}-\eqref{Zcondition}
\begin{equation}\label{map-sol}
 \begin{array}{ll}
\Psi: B_{\frac{\delta}{2}}(\bold{u}_0)\subset H^1(\mathcal G)\cap \mathcal N_{0,Z}(1)&\to X_1(T)\\
\hskip1.2in \bold{j}=(f_0,g_0)&  \mapsto \bold{z}= \Phi(\bold{z}_0)|_{\mathcal G\times [0, T]}
\end{array}
\end{equation}
is of class $C^2$, where $\bold{z}_0=(p_0, q_0)\in B_\delta(\tilde{\bold{u}}_0)$ represents the even-extension of $\bold{j}$, namely, $p_0, q_0$ are even functions on whole of the line with $p_0|_{H^1(\mathbb R^+)}=f_0$, $q_0|_{H^1(\mathbb R^-)}=g_0$ and $\|p_0\|_{H^1(\mathbb R)}\leq 2 \| f_0\|_{H^1(\mathbb R^+)}$, $\|q_0\|_{H^1(\mathbb R)}\leq 2 \| g_0\|_{H^1(\mathbb R^-)}$. We note that we can choose without loss of generality that $\tilde{\bold{u}}_0$ is the par-extension of $\bold{u}_0$. For showing that $\Psi$ in \eqref{map-sol} is of class $C^2$ we will write it as a composition of $C^2$-maps. We start by considering the even-extension mapping 
$$
\mathcal E: B_{\frac{\delta}{2}}(\bold{u}_0)\subset H^1(\mathcal G)\cap \mathcal N_{0,Z}(1)\to B_\delta(\tilde{\bold{u}}_0)
$$
which is well defined because $\|\mathcal E(\bold{j})-\tilde{\bold{u}}_0\|_{H^s(\mathbb R)\times H^s(\mathbb R)}= \|\mathcal E(\bold{j}-\bold{u}_0)\|_{H^s(\mathbb R)\times H^s(\mathbb R)}\leq \sqrt{2}  \|\bold{j}-\bold{u}_0)\|_{H^s(\mathcal G)}<\sqrt{2}\frac{\delta}{2}<\delta$. Moreover, it is not difficult to see that $\mathcal E$ is a Lipschitz mapping of class $C^\infty$. Next, we consider the linear  restriction-mapping 
$$
 \begin{array}{ll}
\mathcal R: X_1&\to X_1(T)\\
\hskip0.3in\bold{w}& \mapsto \bold{w}|_{\mathcal G \times [0,T]}
\end{array}
$$
which is well-defined and continuous. Indeed, we have $\|\mathcal R \bold{w}\|_{X_1(T)}\leq 2 \|\bold{w}\|_{X_1}$. The proof of the former inequality follows by using the restriction-norm associated to every norm in $X_1$. For instance, in  the case of the norm $\|\cdot\|_{C([0, T]; H^1(\mathcal G))}$ we have for $\bold{w}=(f,g)\in X_1$ and  for every $t\in [0,T]$ fixed, $$\|f(t)\|_{H^1(\mathbb R^{+})}=inf\{\|q\|_{H^1(\mathbb R_x)}: q|_{\mathbb R^{+}}=f(t)\}\leq \|f(t)\|_{H^1(\mathbb R_x)}\leq \|\bold{w}\|_{C(\mathbb R_t; H^1(\mathbb R_x))}$$ and $$\|g(t)\|_{H^1(\mathbb R^{-})}\leq \|\bold{w}\|_{C(\mathbb R_t; H^1(\mathbb R_x))}.$$ Therefore, $\|\mathcal R \bold{w}\|_{C([0, T]; H^1(\mathcal G))}\leq 2 \|\bold{w}\|_{X_1}$. Moreover, $\mathcal R$ is a $C^\infty$-mapping. Lastly, since $\Psi= \mathcal R\circ \Phi\circ \mathcal E$ follows that $\Psi$ is of class $C^2$. This finishes the proof.

\end{proof}

\section{Nonlinear instability}\label{6}

The focus of this section is to show the   nonlinear instability result in Theorem \ref{Nins}. For convenience of the reader we give a brief review of the results in  \cite{AC} which will be sufficient in our proof. 

We consider the family of stationary profiles  for the KdV model on a balanced graph  given by  $(\phi_{\bold e}(x))_{\bold e\in \bold E}=U_{Z}=(u_{-}, u_{+})\in D(A_Z)$ with $u_{-}=(\phi_{-})_{\bold e\in \bold E_{-}}$, $u_{+}=(\phi_{+})_{\bold e\in \bold E_{+}}$ defined in \eqref{u+} and $Z\neq 0$. Next, we suppose  for $\bold e\in \bold E$, that $u_{\bold e}$ satisfy formally  the KdV equation in \eqref{kdv3}  and define
\begin{equation}\label{stat3}
	v_{\bold e}(x,t)\equiv u_{\bold e}(x,t) -\phi_{\bold e}(x).
\end{equation}
Then,  for each $\bold e\in \bold E$ we have the equation 
\begin{equation}\label{stat4}
	\partial_t  v_{\bold e}= \alpha_ \bold e \partial_x ^3v_\bold e + \beta_ \bold e\partial_x v_\bold e + 2 \partial_x (v_\bold e \phi_\bold e) +  \partial_x (v^2_\bold e).
\end{equation}
Thus, we have that the system  (abusing the notation)
\begin{equation}\label{stat5}
	\partial_t  v_{\bold e}(x,t)= \alpha_ \bold e \partial_x ^3v_\bold e(x,t) + \beta_ \bold e \partial_x v_\bold e(x,t) + 2 \partial_x (v_\bold e(x,t) \phi_\bold e(x)), 
\end{equation}
represents the linearized equation for the KdV in \eqref{kdv3} around $(\phi_{\bold e}(x))_{\bold e\in \bold E}$. Next, we looking for a {\it growing mode solution} of \eqref{stat5} with the form 
$$
v_{\bold e}(x,t)=e^{\lambda t} \psi_{\bold e}\;\; \text{and}\;\; \text{Re} (\lambda) >0.
$$
In other words, we need to solve the formal system for $\bold e\in \bold E$  (we recall $(\alpha_ \bold e)_{\bold e\in \bold E}=(\alpha_{+})$, $(\beta_ \bold e)_{\bold e\in \bold E}=(\beta_{+}) $, $\alpha_{+}>0$ and $\beta_{+}<0$),
\begin{equation}\label{stat6}
	\lambda \psi_{\bold e}=-\partial_x\mathcal L_{\bold e}  \psi_{\bold e},\qquad \mathcal L_{\bold e}=-\alpha_{+}\partial_x^2-\beta_{+}-2\phi_{\bold e},
\end{equation}
Next, we write our eigenvalue problem in \eqref{stat6} in an Hamiltonian matrix. Indeed, for 
$\psi=(\psi_{-}, \psi_{+})$ with $\psi_{-}=( \psi_{\bold e})_{\bold e\in \bold E_{-}}$ and  $\psi_{+}=( \psi_{\bold e})_{\bold e\in \bold E_{+}}$,  we set $(\mathcal L_{\bold e})_{\bold e\in \bold E}=(\mathcal L_{-}, \mathcal L_{+})$ where
\begin{equation}\label{stat7} 
	\begin{aligned}
		\mathcal L_{-}\psi_{-}&=(-\alpha_{+}\partial_x^2\psi_{\bold e}-\beta_{+}\psi_{\bold e}-2\phi_{\bold e}\psi_{\bold e})_{\bold e\in \bold E_{-}},\\
		\\
		\mathcal L_{+}\psi_{+}&=(-\alpha_{+}\partial_x^2\psi_{\bold e} -\beta_{+}\psi_{\bold e}-2\phi_{\bold e}\psi_{\bold e})_{\bold e\in \bold E_{+}}.
	\end{aligned}
\end{equation}
Thus, eigenvalue problem in \eqref{stat6} can be written in a Hamiltonian vectorial form
\begin{equation}\label{stat8} 
	\lambda\left(\begin{array}{c} \psi_{-} \\ \psi_{+}\end{array}\right) =\left(\begin{array}{cc} -\partial_x \mathcal L_{-}& 0 \\0 & -\partial_x \mathcal L_{+}\end{array}\right)\left(\begin{array}{c} \psi_{-}\\ \psi_{+}\end{array}\right)\equiv NE\left(\begin{array}{c} \psi_{-}\\ \psi_{+}\end{array}\right)
\end{equation}
where we are  identifying $ \mathcal L_{\pm}$ as a  $n\times n$-diagonal matrix and $N$, $E$  are  $2n\times 2n$-diagonal matrix defined by
\begin{equation}\label{stat10} 
	N=\left(\begin{array}{cc} -{\bf{\partial_x}}I & 0 \\ 0  & -{\bf{\partial_x}}I \end{array}\right),\quad E=\left(\begin{array}{cc} \mathcal L_{-}& 0 \\0 & \mathcal L_{+}\end{array}\right),
\end{equation} 
with $I$ being the $n\times n$-identity matrix and $\bf{\partial_x}$ the $n\times n$-diagonal matrix ${\bf{\partial_x}}=\text{diag}\Big(\partial_x, ..., \partial_x)$

If we denote by $\sigma(NE)=\sigma_p(NE)\cup \sigma_{ess}(NE)$ the spectrum  of $NE$ (namely, $\lambda \in \sigma_p(NE)$ if $\lambda$ is isolated with finite multiplicity, and $\sigma_{ess}(NE)$ is the essential spectrum), the later discussion
suggests the usefulness of the following definition:

\begin{definition}
	The stationary vector solution $(\phi_{\bold e})_{\bold e\in \bold E}\in D(H_Z)$    is said to be \textit{linearly  stable} for model \eqref{kdv3} if the spectrum of $NE$, $\sigma(NE)$, satisfies $\sigma(NE)\subset i\mathbb{R}.$
	Otherwise, the stationary solution $(\phi_{\bold e})_{\bold e\in \bold E} $ is said to be \textit{linearly unstable}.
\end{definition}

It is standard to show that $\sigma(NE)$ is symmetric with respect  to both the real and imaginary axes and $ \sigma_{ess}(NE)\subset i\mathbb{R}$ by supposing $N$ skew-symmetric and $E$ self-adjoint   (see, for instance, \cite[Lemma 5.6 and Theorem 5.8]{GST}). Thus, by considering for $\bold{u}=(u_ \bold e)_{\bold e\in \bold E}$ the  notation
\begin{equation*}
	\bold{u}=( u_{1, -},...,u_{n, -}, u_{1, +},...,u_{n, +}),
\end{equation*} 
we define  the set of elements of $L^2(\mathcal G)$ continuous at the graph-vertex $\nu=0$ as
\begin{equation}\label{conti} 
	\mathcal C =\{\bold{u} \in L^2(\mathcal G): u_{1, -}(0-)=...=u_{n, -}(0-)=u_{1, +}(0+)=...=u_{n, +}(0+)\}.
\end{equation}
Then from Lemma 6.4 and Proposition 7.4  in  \cite{AC}, we obtain that for the following domain
\begin{equation}\label{E_0}
	D( E )= \Big\{u\in H^2(\mathcal G) \cap \mathcal C:   \sum_{\bold e \in \bold E_{+}} u'_{\bold e}(0+)-\sum_{\bold e \in \bold E_{-}} u'_{\bold e}(0-)=Zn u_{1, +}(0+)\Big\},
\end{equation} 
$E$ in \eqref{stat10} is a self-adjoint operator. Moreover, $N$ in  \eqref{stat10} will be skew-symmetric in the domain $H^1(\mathcal G) \cap \mathcal C$. Hence, by the comments above  follows that will be equivalent to say that $(\phi_{\bold e})_{\bold e\in \bold E}\in D(A_Z)$ is  \textit{linearly stable} if $\sigma_p(NE)\subset i\mathbb{R}$, and it is linearly unstable if $\sigma_p(NE)$ contains points $\lambda$ with  $\text{Re} (\lambda)>0.$ We note that $D(A_Z)\cap \mathcal C\subset D( E )$ and $(\phi_{\bold e})_{\bold e\in \bold E}=U_{Z}\in D( E )$.  

Thus from Theorems 4.4, 4.6, and 6.1 in  \cite{AC}, we have the following linear instability result.

\begin{theorem}\label{ins}  Let $Z\neq 0$.  For $\alpha_{+}>0$,  $\beta_{+}<0$, and $-\frac{\beta_{+}}{\alpha_+}>\frac{Z^2}{4}$, we consider the profiles $\phi_{\pm}$  in \eqref{soli6}.  Define $U_{Z}= (\phi_{\bold e})_{\bold e\in \bold E}\in D(A_Z)$ with $\phi_{\bold e}= \phi_{-}$ for $\bold e\in \bold{E}_{-}$ and   $\phi_{\bold e}= \phi_{+}$ for $\bold e\in \bold{E}_{+}$. Then, 
	$$
	\Phi_{Z}(x,t)=U_{Z}(x)
	$$ 
	defines a family  of linearly unstable stationary solutions for the  Korteweg-de Vries equation in \eqref{kdv3}. Moreover, the operator $NE$ in \eqref{stat10} has a real positive eigenvalue, namely, there is $\lambda>0$, $\Psi_\lambda\in D(A_Z)\cap \mathcal C$, $\Psi_\lambda\neq 0$, such that $NE \Psi_\lambda=\lambda \Psi_\lambda$.
\end{theorem}

The strategy of proof for Theorem \ref{Nins} is to use the linear instability result  in Theorem \ref{ins}, the approach by Henry {\it {et al.}} \cite{HPW82}, and that the mapping data-solution associated to the IVP in \eqref{kdv3} is of class $C^2$ on $H^1(\mathcal G)\cap \mathcal N_{0,Z}(1)$ (see Theorem \ref{dependence} and \eqref{trace0}). For convenience of the reader, we establish the following theorem which is the link to obtain nonlinear instability from a linear instability result (see \cite{HPW82, ANGNATI2016,AngLop08}).

\begin{theorem}[Henry \emph{et al.} \cite{HPW82}]
	\label{henry1} Let $Y$ be a Banach space
	and $\Omega\subset Y$ an open set containing $0$. Suppose
	$\mathcal T:\Omega\to Y$ has $\mathcal T(0)=0$, and for some $p>1$ and continuous
	linear operator $\mathcal L$ with spectral radius $r(\mathcal L)>1$
	we have
	$
	\|\mathcal T(x)-\mathcal Lx\|_Y=O(\|x\|_Y^p)\;\;as\;\;x\to 0.
	$
	Then $0$ is unstable as a fixed point of $\mathcal T$.
\end{theorem}

\begin{remark}
	\label{remHenry}
	The statement in Theorem \ref{henry1} establishes the instability of $0$ as a fixed point of $\mathcal T$; in other words, it shows the existence of points moving away from $0$ under successive applications of $\mathcal T$
\end{remark}

Theorem \ref{henry1} can be recast in a more suitable form for applications to nonlinear wave instability (see also \cite{AngNat2014,ANGNATI2016}).

\begin{corollary}\label{corohen} Let $S:\Omega_1\subset Y\to Y$ be a
	$C^{2}$ map defined in an open neighborhood of a fixed point $\phi$
	of $S$. If there is an element $\mu\in \sigma(S'(\phi))$ with
	$|\mu|>1$ then $\phi$ is an unstable fixed point of $S$.
\end{corollary}
\begin{proof}  For $x\in \Omega\equiv\{y-\phi: y\in \Omega_1\}$ we consider the mapping $\mathcal T(x)\equiv S(x+\phi)-\phi$. Then, clearly, $\mathcal T(0)=0$ and $\mathcal T$ is of class $C^2$ in  $\Omega$. Define $\mathcal D\equiv S'(\phi)$. Then, by hypothesis, there is an eigenvalue $\mu\in \sigma(\mathcal D)$ with $1<|\mu|\leq r(\mathcal D)$. By Taylor's formula  
	$$ \mathcal T(x)=\mathcal T(0)+\mathcal T'(0)x+ O(\|x\|_Y^2)=\mathcal D
	x+O(\|x\|_Y^2)
	$$ 
	provided that $\|x\|_Y \ll 1$. Then Theorem \ref{henry1} implies the existence of $\epsilon_0 > 0$ such that, for any ball $B_\eta(\phi)$, with radius $\eta > 0$ and arbitrarily large $N_0 \in \mathbb N$, there exists $n \geq N_0$ and $y \in B_\eta(\phi)$ such that $\| S^n(y) - \phi \|_Y \geq \epsilon_0 $. This completes the proof.
\end{proof}

Before proving our Theorem \ref{Nins}, we need to specify the particular mapping $S$ in Corollary \eqref{corohen} suitable for our needs. We start by giving the following notation,  the unique solution $\bold{u}$ to the Cauchy problem \eqref{kdv3}-\eqref{Zcondition} with initial datum $\varphi\in H^1(\mathcal G)\cap \mathcal N_{0,Z}(1)$ given by Theorem \ref{theorem1} will denoted as $\bold{u}=\Upsilon(\varphi)\in X_1(T)\subset C([0,T]: H^1(\mathcal G)\cap \mathcal N_{0,Z}(1))$. The former relation follows from the first relation in  \eqref{Zcondition}  and  Sobolev's embedding.


Let us now define a mapping which plays the role of the operator $ S$ in the abstract Corollary \ref{corohen}. Let $\Phi=(\phi_{\bold e})_{\bold e\in \bold E}$ be  a stationary profile for equation \eqref{kdv3},  and $B$  the ball $B=B_\epsilon (\Phi) = \{ \varphi \in H^1(\mathcal G)\cap \mathcal N_{0,Z}(1) \, : \, \| \varphi - \Phi \|_1 < \epsilon \}$ with $\epsilon > 0$.  For each $\varphi \in B$, set
\begin{equation}\label{defS}
	\begin{aligned}
		S &: B \to H^1(\mathcal G)\cap \mathcal N_{0,Z}(1),\\
		S(\varphi) &:= \Upsilon(\varphi)(1)
	\end{aligned}
\end{equation}
where $\Upsilon(\varphi)\in C([0,{1}];  H^1(\mathcal G)\cap \mathcal N_{0,Z}(1))$ is the unique solution to the Cauchy problem \eqref{kdv3}-\eqref{Zcondition}  with $\Upsilon(\varphi)(0)=\varphi$. We are choosing $T=1$ in Theorem \ref{theorem1} without loss of generality. We recall that by the continuity property of the mapping data-to-solution given in Theorem \ref{theorem1},  we can choose $\epsilon $ small enough such that the solution $\Upsilon(\varphi)$ for every $\varphi\in B_\epsilon (\Phi) $ can be defined on $[0,1]$ because  the stationary solution $\Phi$ is defined for all $t\in \mathbb R$. Thus, $S$ is a well-defined mapping.

\begin{lemma}[properties of $S$]
	\label{propS} 
	Let $\Phi=(\phi_{\bold e})_{\bold e\in \bold E}\in D(A_Z)$ be a stationary profile for equation the KdV model in\eqref{kdv3}. The mapping $S$ defined in \eqref{defS} satisfies:
	\begin{itemize}
		\item[(a)] $S(\Phi) = \Phi$.
		\item[(b)] $S$ is twice Fr\'echet differentiable in an open neighborhood of $\Phi$.
		\item[(c)] For every $\varphi \in H^1(\mathcal G)\cap \mathcal N_{0,Z}(1))$ there holds
		\begin{equation}
			\label{devS}
			S'(\Phi) \varphi = \bold{V}_\varphi(1),
		\end{equation}
		where $\bold{V}_\varphi \in X_1(1) $, denotes the unique solution to the following linear Cauchy problem 
		\begin{equation}\label{Cplin}
			\left\{\begin{aligned}
				\frac{d}{dt}\bold{V}(t) &= NE\bold{V}(t),\quad t\in (0,1),\\
				\bold{V}(0) &= \varphi,
			\end{aligned}\right.
		\end{equation}
		with $N$ and $E$ defined in \eqref{stat7}-\eqref{stat10}.		
	\end{itemize}
\end{lemma}
\begin{proof}

	First, it is obvious that  $S(\Phi) = \Phi$.  Now, from Theorem \ref{dependence} (by choosing $\epsilon$ small enough again)  we know that the data-to-solution map $\varphi \mapsto \Upsilon(\varphi)\in  X_1(1)$ is of class $C^2$. Hence, $S$ is twice Fr\'echet differentiable on $B_{\epsilon}(\Phi)$. This proves (b).
	
Next, we obtain the Fr\'echet derivative's formula in \eqref{devS}  by computing the G\^ateaux derivative, namely,
\begin{equation}\label{frechet}
	S'(\Phi) \varphi = \frac{d}{d\epsilon} S(\Phi + \epsilon \varphi) \big|_{\epsilon = 0}= \bold{V}_\varphi(1),
\end{equation}	
	for any arbitrary $\varphi \in  H^1(\mathcal G)\cap \mathcal N_{0,Z}(1)$.  The proof of the equality in \eqref{frechet} is standard and it is  essentially based on a good representation of the solution $\Upsilon(\varphi)$. Some calculations made {\it{ a priori}} and based in the representation induced by the relations \eqref{aux1}-\eqref{aux2}
show us that is possible to obtain the relation in  \eqref{frechet}, but will require long calculations based on strategics of extensions and restrictions of solutions. Thus, by convenience of the reader, we will take advantage of the results established in Propositions  7.8 and 7.10 in Angulo\&Cavalcante \cite{AC} about the unitary group $\{W(t)\}_{t\in \mathbb R}$ generated by $(A_Z, D(A_Z))$ in \eqref{domain8}. In particular, we get the invariance of $H^1(\mathcal G)\cap \mathcal N_{0,Z}(1)$ by $W(t)$. Now,  we observe by definition that  $S(\Phi + \epsilon \varphi) =  \Upsilon(\Phi + \epsilon \varphi) (1)$ and since $\Upsilon$ is of class $C^2$ around $\Phi$ on the Banach space $X_1(1)$ we make the Taylor's expansion
	\begin{equation}
		\label{exp1}
	 \Upsilon (\Phi + \epsilon \varphi) = \Upsilon(\Phi) + \epsilon \Upsilon'(\Phi) \varphi + O(\epsilon^2).
	\end{equation}
Next, we determine	an expression for $\Upsilon'(\Phi) \varphi$ (we recall that $\Upsilon'(\Phi)\in \mathcal L( H^1(\mathcal G)\cap \mathcal N_{0,Z}(1);X_1(1)$)).  Indeed, by using the Duhamel's integral representation for $ \Upsilon (\Phi + \epsilon \varphi)$ we know that
	$$
	\Upsilon(\Phi + \epsilon \varphi)(t) = W(t)(\Phi + \epsilon \varphi) +2\int_0^t W(t-\tau) \Upsilon(\Phi + \epsilon \varphi)(\tau)\partial_x \Upsilon(\Phi + \epsilon \varphi)(\tau) \, d\tau, \qquad \text{for}\;\;t\in [0,1]
	$$
	then from \eqref{exp1}
	\begin{equation}\label{exp2}
	\Upsilon(\Phi + \epsilon \varphi)(t) = W(t)\Phi + 2\int_0^t W(t-\tau) [ \Upsilon(\Phi) \partial_x\Upsilon(\Phi)](\tau)\, d\tau+\epsilon V_{\Phi, \varphi}(t)+ O(\epsilon^2), \qquad \text{for}\;\;t\in [0,1]
	\end{equation}
	where 
$$
	V_{\Phi, \varphi}(t)= W(t)\varphi + 2\int_0^t W(t-\tau) [ \partial_x(\Upsilon(\Phi) (\Upsilon'(\Phi)\varphi)(\tau))]\, d\tau, \qquad \text{for}\;\;t\in [0,1].
	$$
	\begin{equation}\label{exp3}
		V_{\Phi, \varphi}(t)= (\Upsilon'(\Phi) \varphi)(t)\;\; \text{for all} \;\; t\in [0,1],
	\end{equation}
and so $V_{\Phi, \varphi} \in X_1(1)$  and satisfies the integral equation
$$
	V_{\Phi, \varphi}(t)= W(t)\varphi + 2\int_0^t W(t-\tau) [ \partial_x(\Upsilon(\Phi) V_{\Phi, \varphi} (\tau))]\, d\tau,
	$$
	with $V_{\Phi, \varphi}(0)=\varphi$. Then, since $\Upsilon(\Phi)=\Phi$ and for $\zeta=(\zeta_\bold e)$
	$$
	NE\zeta=A_Z\zeta+ 2\text{diag}((\partial_x(\phi_\bold e\zeta_\bold e))\delta_{ij}),\quad 1\leq i,j\leq |\bold E_{-}|+ |\bold E_{+}|,
$$
follows that  $V_{\Phi, \varphi}$ is the solution to the linearized Cauchy problem \eqref{Cplin} in the distributional sense. Lastly, from \eqref{frechet}, \eqref{exp1}, and \eqref{exp3} we get $S'(\Phi) \varphi =\bold{V}_{\Phi, \varphi}(1)$, for any $\varphi\in H^1(\mathcal G)\cap \mathcal N_{0,Z}(1)$. This shows (c) and the lemma is proved.
\end{proof}

We are now able to prove our main instability result.

\begin{proof} $[${\bf{Theorem \ref{Nins}}}$]$
	
	Let us consider the eigenfunction $\Psi_\lambda$ of the linearized operator $NE$ in \eqref{stat10} associated to the positive eigenvalue $\lambda$ given by Theorem \ref{ins}. Then, since $\Psi_\lambda\in D(NE)\cap \mathcal C=D(A_Z)\cap\mathcal C$ follows $\Psi_\lambda\in H^1(\mathcal G)\cap \mathcal N_{0,Z}(1)$. Moreover,  for $\bold{V}(t) = e^{\lambda t} \Psi_\lambda$, $\bold{V}\in C(\mathbb R ; H^3(\mathcal G)\cap \mathcal N_{0,Z}(1))$ and satisfies 
	$$
	\partial_t \bold{V}(t) = \lambda e^{\lambda t} \Psi_\lambda= e^{\lambda t} NE \Psi_\lambda = NE \big( e^{\lambda t} \Psi_\lambda \big) = NE \bold{V}(t),
	$$
	with $\bold{V}(0)=\Psi_\lambda$. Furthermore,  $\bold{V}\in X_1(1)$ in \eqref{X}. In fact, we have that  for $\Psi_\lambda=(\psi_{\bold e})_{\bold e\in \bold E}\in D(A_Z)$, then $\partial^3_x {\psi}_{\bold e}\in L^2(I)$, $I=(-\infty,0)$ or  $I=(0, \infty)$, and also by choosing a cutoff regular  function $\phi(t):\mathbb R\rightarrow \mathbb R$ supported on the set $[-2,2]$, such that $\phi\equiv 1$ on the set $[-1,1]$ we have that each coordinate function of $\phi(t)\bold V$ given by $\phi(t)e^{\lambda t}\psi_{\bold e}(x)$ is in $X^{1,b,\beta,\sigma}(\R^2)$. It follows immediately that $\bold{V}\in X_1(1)$.
	
Now, define $\mu := e^{\lambda}$. This yields by Lemma \ref{propS},
	$$
	S'(\Phi) \Psi_\lambda = \bold{V}  (1) = e^{\lambda} \Psi_\lambda = \mu \Psi_\lambda.
	$$
	This shows that $\mu \in \sigma(S'(\Phi))$ with $|\mu| > 1$ because $\lambda > 0$. Thus, the mapping defined in \eqref{defS} on an open neighborhood of $\Phi$ satisfies the hypotheses of Corollary \ref{corohen}. Therefore, the profiles $\Phi=U_Z$ of either tail or bump type are nonlinearly unstable by the flow of the Korteweg-de Vries model in $H^1(\mathcal G)$-norm. The proof is complete.
\end{proof}

 \vskip0.1in

\noindent
{\bf Acknowledgements.}  J. Angulo was partially funded by CNPq/Brazil Grant and Pronex-FAPERJ. Furthermore,  it author would like to thank to Universidade Federal de Alagoas (UFAL), Macei\'o/Brazil,  by the support and the warm stay where part of the present project was developed. M. Cavalcante  was supported by CNPq 310271/2021-5 and  the Funda\c c\~ao de Amparo \`a Pesquisa do Estado de Alagoas
- FAPEAL, Brazil (Process E:60030.0000000161/2022).

\vskip0.1in
 \noindent
{\bf Data availability statement}. The data that support the findings of this study are available upon request from the authors.

\end{document}